\documentclass[10pt]{amsart}
\usepackage{amssymb}
\usepackage{amscd}

\newcommand{\bZ}{{\mathbb Z}}

\newcommand{\bC}{{\mathbb C}}

\newcommand{\bQ}{{\mathbb Q}}

\newcommand{\bT}{{\mathbb T}}

\newcommand{\bG}{{\mathbb G}}

\newcommand{\II}{{I_{\infty}^2}}

\newtheorem{thm}{Theorem}[section]
\newtheorem{lemma}[thm]{Lemma}
\newtheorem{cor}[thm]{Corollary}
\newtheorem{prop}[thm]{Proposition}

\numberwithin{equation}{section}

\begin{document}

\title[Chow rings of flag vaieties ]{Chow rings of 
versal 
complete flag varieties}
 
\author{Nobuaki Yagita}

\address{ faculty of Education, 
Ibaraki University,
Mito, Ibaraki, Japan}
 
\email{nobuaki.yagita.math@vc.ibaraki.ac.jp, }


\maketitle

\begin{abstract}
In this paper, we compute Chow rings of generically
twisted (versal) complete flag varieties  corresponding to
some  simple Lie groups.
\end{abstract}

\section{Introduction}

Let $G$ and  $T$ be a  connected compact Lie group
and its maximal torus.  Let 
$G_k$ and $T_k$ be a split  reductive group and split maximal torus over a field $k$ with $ch(k)=0$, 
corresponding to  Lie groups  $G$ and  $T$.  Let $B_k$ be the Borel
subgroup containing $T_k$.

Moreover we take $k$
% (some extension of a first given field)
such that there is a $G_k$-torsor
$\bG_k$ which is isomorphic to a versal $G_k$-torsor ( for the definition of a versal $G_k$-torsor, see $\S 4$ below or
see [Ga-Me-Se], [To2], [Me-Ne-Za], [Ka1]).  Then  
$X=\bG_k/B_k$ is thought as the $most$ $twisted$  
complete flag variety.  (We say that such $X$ is
 a generically  twisted
or a versal  flag variety  [Me-Ne-Za], [Ka1].) 
In this paper, we study the Chow ring $CH^*(X)$ for the
versal flag variety $X=\bG_k/B_k$.

Fix a prime number $p$.
By Petrov-Semenov-Zainoulline ([Pe-Se-Za], [Se-Zh], [Se]), it is known that the $p$-localized motive
$M(X)_{(p)}$ of $X$ is decomposed as
\[ M(X)_{(p)}=M(\bG_k/B_k)_{(p)}\cong \oplus_i R(\bG_k)\otimes\bT^{\otimes s_i}\]
where $\bT$ is the Tate motive and $R(\bG_k)$ is  some
motive called generalized Rost motive. (It is the original 
Rost motive([Ro1,2], [Vo2,3]) when $G$ is of type $(I)$ as explained below).
Hence we have maps
\[ CH^*(BB_k)\to CH^*(X)\stackrel{split\ surj.}{\twoheadrightarrow} CH^*(R(\bG_k))\]
where $BB_k$ is the classifying space for $B_k$-bundles.

{\bf Remark.}  In this paper, a map $A\to B$ (resp. $A\cong B$) for rings $A,B$
means a ring map (resp. a ring isomorphism).
  However $CH^*(R(\bG_k))$ does not have a
 canonical ring structure.
 Hence a map $A\to CH^*(R(\bG_k))/p$ (resp. $A\cong CH^*(R(\bG_k))/p$)
means only a (graded) additive map 
(resp. additive isomorphism) even if $CH^*(R(\bG_k))/p$ has some ring structure.  For example, the
above first map is a ring map but the second is not a ring map.
 
From Merkurjev and Karpenko [Me-Ne-Za], [Ka1], we know
that the first map is also surjective when $\bG_k$ is a versal
$G_k$-torsor. We study 
what elements in $CH^*(BB_k)\cong CH^*(BT_k)$ 
generate $CH^*(R(\bG_k))$.

For example, Petrov (Theorem 1 in  [Pe]) computed the integral $CH^*(Y)$
for the versal maximal orthogonal Grassmannian $Y$
corresponding to $G=SO(2\ell+1)$, $\ell>0$.
It is torsion free and is isomorphic to $CH^*(R(\bG_k))$
 (see Theorem 7.13 below).
Hence the restriction map $CH^*(X)\to CH^*(G_k/B_k)$ is injective.  Thus  we know the ring structure of 
 $CH^*(X)$ from  that of $CH^*(G_k/B_k)$ ([El-Ka-Me], [Vi], [Tod-Wa]).
These Petrov's results can be very simply written,
when we consider the $mod(2)$ Chow theories.
\begin{thm}  Let $(G,p)=(SO(2\ell+1),2)$ and $X=\bG_k/B_k$ be a versal flag variety.  Then there are isomorphisms
\[ CH^*(R(\bG_k))/2\cong  \bZ/2[c_1,...,c_{\ell}]/(c_1^2,...,c_{\ell}^2)=\Lambda(c_1,...,c_{\ell}),\]
\[ CH^*(X)/2\cong S(t)/(2,c_1^2,...,c_{\ell}^2)\]
where $c_i=\sigma_i(t_1,...,t_{\ell})$ is the $i$-th elementary symmetric
function in 
\[S(t)=CH^*(BB_k)\cong H^*(BT)\cong \bZ[t_1,...,t_{\ell}].\]
%(i.e., $i$-th Chern class for the representation $T\subset
%U(\ell)$).
\end{thm}
{\bf Remark.}  We have an isomorphism $CH^*(X)/2\cong H^*(Sp(\ell)/T;\bZ/2)$
for the symplectic group $Sp(\ell)$ (see Corollary 7.9).

We give a new proof of the above theorem,  which can be 
worked for other groups such that Chow rings $CH^*(X)$ have $p$-torsion elements. The additive structures in the following theorem are known ([Me-Su], [Ya4], [Ka-Me]).
However, the ring structure of $CH^*(X)/p$ was unknown
except for $(G,p)=(G_2,2)$ ([Ya3]). 
\begin{thm} 
Let $G$ be  of type $(I)$ and $rank(G)=\ell$.
Then $2p-2 \le \ell$, and we can take $b_i\in CH^*(BB_k)$ for $1\le i\le \ell$
such that there are 
 isomorphisms
\[ CH^*(R(\bG_k))/p\cong \bZ/p\{b_1,...,b_{2p-2}\},\]
\[ CH^*(X)/p\cong S(t)/(p, b_ib_j,b_k|0\le i,j\le 2p-2<k\le \ell)\]
where  $\bZ/p\{a,b,...\}$ is the $\bZ/p$-free module generated by $a,b,...$
Moreover the ideal of 
torsion elements in $CH^*(X)_{(p)}$ is generated by $b_1,b_3$,...,$b_{2p-3}$.
\end{thm}
Here $b_i\in H^*(BT)$ are transgression
images in the spectral sequence induced
from the fibering $G\to G/T\to BT$.
These $b_i$ are explicitly known ([Tod2], [Tod-Wa], [Na]),
for example, when $(G,p)=(G_2,2)$, we can take
$b_1=t_1^2+t_1t_2+t_2^2$ and $b_2=t_2^3$ in
$H^*(BT)\cong \bZ[t_1,t_2]$ with $|t_i|=2$.

To explain the transgression and type $(I)$ groups,
 we recall how to compute
$H^*(G/T)$ in algebraic topology. 
By Borel, its $mod(p)$ cohomology is (for $p$ odd)
\[ H^*(G;\bZ/p)\cong P(y)/p\otimes \Lambda(x_1,...,x_{\ell}),
\quad |x_i|=odd\]
where  $P(y)$ is a truncated polynomial ring 
generated by $even$ dimensional elements $y_i$.  When $p=2$, we consider the  graded ring $grH^*(G;\bZ/2)$ which is isomorphic to the right  hand side ring above.

When $G$ is simply connected and $P(y)$ is generated by just one generator, 
we say that $G$ is of type $(I)$.  Except for 
$(E_7,p=2)$ and $(E_8,p=2,3)$, all exceptional (simple) Lie groups are of type $(I)$.  Note that in these cases, it is known
$rank(G)=\ell\ge 2p-2$.

We consider the fibering ([Tod2], [Mi-Ni], [Na])
$G\stackrel{\pi}{\to}G/T\stackrel{i}{\to}BT$
and the induced spectral sequence 
\[ E_2^{*,*}=H^*(BT;H^*(G;\bZ/p)) \Longrightarrow H^*(G/T;\bZ/p).\] 
Here we  can write 
$H^*(BT)\cong S(t)=\bZ[t_1,...,t_{\ell}]$
with $|t_i|=2.$

It is well
known that $y_i\in P(y)$ are permanent cycles and 
that there is a regular sequence 
$(\bar b_1,...,\bar b_{\ell})$ in $H^*(BT)/(p)$ such that $d_{|x_i|+1}(x_i)=\bar b_i$ ([Tod2], [Mi-Ni]). 
The element $\bar b_i$ is called the transgression image of $x_i$.
We know that $G/T$ is a manifold such that
$H^*(G/T)=H^{even}(G/T)$ and $H^*(G/T)$ is torsion free.
We also see that 
there is a filtration in $H^*(G/T)_{(p)}$ such that  
\[ grH^*(G/T)_{(p)}\cong P(y)\otimes S(t)/(b_1,...,b_{\ell})\]
where $b_i\in S(t)$ with $ b_i=\bar  b_i\ mod(p)$.

The transgression images $b_i$ in Theorem 1.2
are just $b_i$ above. When $(G,p)=(SO(2\ell+1),2)$
we can take $b_i=c_i$. 
In fact, $b_i\ mod(p)$ are generators  (such that 
$|b_i|\le |b_j|$ for $i<j$ in most cases) of the kernel $I(p)$ of the map
$H^*(BT)/p \to H^*(G/T)/p$ (it is also isomorphic to the kernel of  $CH^*(BB_k)/p
\to CH^*(G_k/B_k)/p$).

By giving the filtration on $S(t)$ by $b_i$, we 
can write 
\[gr S(t)/p\cong A\otimes S(t)/(b_1,...,b_{\ell})\quad
for \ A=\bZ/p[b_1,...,b_{\ell}].\]
In particular, we have maps
$ A\stackrel{i_A}{\to} CH^*(X)/p\to CH^*(R(\bG_k))/p.$
We easily see that $i_A(A)\supset CH^*(R(\bG_k))/p$.
In particular the above composition map is surjective.
Suppose that there are $f_1(b),...,f_s(b)\in A$ such that 
$CH^*(R(\bG_k))/p\cong A/(f_1(b),...,f_s(b))$. Moreover if $f_i(b)=0$
also in $CH^*(X)/p$,  then we have the isomorphism
\[CH^*(X)/p\cong S(t)/(f_1(b),...,f_s(b)).\]
The first isomorphism of Theorem 1.1 (resp. Theorem 1.2 when  
$\ell=2p-2$)
is rewritten 
\[ CH^*(R(\bG_k))/2\cong A/(I(2)^{[2]}),\quad
(resp.\ \ CH^*(R(\bG_k))/p\cong A/(I(p)^2)\]
where $I(2)^{[2]}=Ideal(x^2|x\in I(2)).$

For other simply connected simple groups
(with $p$-torsion in $H^*(G)$),
it seems that almost nothing was  known for $CH^*(R(\bG_k))/p$ when  $*>3$.
Hence we write down the fundamental facts here.
\begin{thm} Let $(G,p)=(SO(2\ell+1),2),\
(G',p)=(Spin(2\ell+1),2)$, and $\pi:G'\to G$ be the natural projection.  Let $c_i'=\pi^*(c_i)$.  Then 
$\pi^*$ induces maps such that their composition map
is surjective
\[ CH^*(R(\bG_k)/(2,c_1)\cong 
 \Lambda(c_2,...,c_{\ell})\stackrel{\pi^*}{\to}
CH^*(R(\bG_k'))/2 \twoheadrightarrow
\bZ/2\{1,c_2',...,c_{\bar \ell}'\}\]
where $\bar \ell=\ell-1$ if $\ell=2^j$ for some $j>0$,
otherwise $\bar \ell=\ell$.
Moreover $c_{2^k}'-2c_1^{2^k}$, $k>0$ are torsion elements 
 in $CH^*(X)_{(2)}$.
\end{thm}
The right hand side module in the above  
is an important part in $CH^*(R(\bG_k))/2$.
For example, the groups $Spin(7), Spin (9)$ are of type  $(I)$ and $CH^*(R(\bG_k))/2\cong
\bZ/2\{1,c_2',c_3'\}$.  However, the group $Spin(11)$ is not of type $(I)$.
\begin{lemma}  For $(G',p)=(Spin(11),2)$, we have
the surjection
\[ CH^*(R(\bG_k'))/2\twoheadrightarrow  
\bZ/2\{1, c_2',c_3',c_4',c_5',c_2'c_4',c_1^8\}\]
\end{lemma}
{\bf Remark.}  Quite recently, Karpenko [Ka2] proved the above 
surjection is an isomorphism.

\begin{thm} Let $(G,p)=(E_7,2)$, $(E_8,2)$ or $(E_8,3)$
so that $\ell=7$ for $E_7$ and $\ell=8$ for $E_8$.
Then
we have the 
surjective map
\[CH^*(R(\bG_k))/p \twoheadrightarrow
\bZ/p\{1,b_1,...,b_{\ell}\}.\]
Moreover  for $(G,p)=(E_7,2),\ (E_8,3)$, we have
\[ (CH^*(R(\bG_k))/(Tor_p))\otimes \bZ/p\cong
\bZ/p\{1,b_2,...,b_{\ell}, b_2b_{\ell}\}\]
where $Tor_p$ is the submodule of $CH^*(R(\bG_k))_{(p)}$
generated by torsions.
\end{thm}
Note that the above $b_{i}\not =0$ is not 
a trivial fact.
Indeed $b_i=0$ for $2p-2<i\le \ell$ for groups of type $(I)$.

To see the above elements are nonzero, we mainly 
use
the torsion index $t(G)_{(p)}$.
For $dim(G/T)=2d$, the torsion index is defined as
\[ t(G)=|H^{2d}(G/T;\bZ)/H^{2d}(BT;\bZ)|.\]
Let $n(\bG_k)$ be the greatest common divisor of the degrees of all finite field extension $k'$ of $k$ such that $\bG_k$ 
becomes trivial over $k'$.  Then by Grothendieck [Gr], it is known that $n(\bG_k)$ divides $t(G)$.  Moreover, when
$\bG_k$ is a versal $G_k$-torsor, we have $n(\bG_k)=t(G)$
([To2], [Ga-Me-Se]).
Totaro determined [To1,2]) torsion indexes for all simply connected compact Lie groups $G$.  For example,
$t(E_8)=2^63^25$.  

For all exceptional simple groups $G$, we give another proofs of Totaro's results by using arguments of the above transgression images $b_i$ (e.g., Lemma 11.11).  However we can not compute
$t(G)$ for $G=Spin(2\ell+1)$ by our arguments.

We also consider a field $K$ of an extension of $k$ 
such that $R(\bG_k)|_K$ is a direct sum of the original
Rost motives, and study the restriction map
$CH^*(R(\bG_k))/p\to CH^*(R(\bG_k)|_K)/p$
(Theorem 7.12, 11.13, Proposition 10.8, 12.8).
The first two theorems relate to recent results
by Smirnov-Vishik [Sm-Vi] and Semenov [Se]
respectively.

The plan of this paper is the following.
In $\S 2,\S 3$, we recall and prepare the topological arguments for $H^*(G/T)$
and $BP^*(G/T)$.  In $\S 4$, we recall the decomposition
of the motive of a versal flag variety.
 In $\S 5$, we recall the torsion
index briefly. In $\S 6$, we study $U(m), Sp(m)$ and
$PU(p)$ for each $p$. 
In $\S 7,\S8$ we study in the cases  $SO(m)$
and $Spin(m)$ for $p=2$.  
 In $\S 9$, we study the cases that
$G$ is of type $(I)$.  In $\S10, \S 11, \S 12$, we study the cases $(G,p)=(E_8,3),(E_8,2)$ and $(E_7,2)$ respectively.

 \section{Lie groups $G$ and the  flag manifolds $G/T$}  

Let $G$ be a connected 
 compact Lie group.
By Borel, its $mod(p)$ cohomology is (for $p$ odd)
\[ (2.1)\quad H^*(G;\bZ/p)\cong P(y)/p\otimes \Lambda(x_1,...,x_{\ell}),
\quad \ell=rank(G)\]
\[ \quad  with  \ \  P(y)=\bZ_{(p)}[y_1,...,y_k]/(y_1^{p^{r_1}},...,
y_k^{p^{r_k}})
\]
where the degree $|y_i|$ of $y_i$ is even and 
$|x_j|$ is odd.  When $p=2$, a graded ring $grH^*(G;\bZ/2)$ is isomorphic to the
right  hand side ring, e.g. $x_j^2=y_{i_j}$ for some $y_{i_j}$.
In this paper, $H^*(G;\bZ/2)$ means this $grH^*(G;\bZ/2)$
so that $(2.1)$ is satisfied also for $p=2$.

 Let $T$ be the  maximal torus of $G$.  and $BT$ be the classifying space of $T$.
We consider the fibering ([Tod2], [Mi-Ni])
$ G\stackrel{\pi}{\to}G/T\stackrel{i}{\to}BT$
and the induced spectral sequence 
\[ E_2^{*,*'}=H^*(BT;H^{*'}(G;\bZ/p)) \Longrightarrow H^*(G/T;\bZ/p).\] 
The cohomology of the classifying space of the torus is  given by
$H^*(BT)\cong S(t)=\bZ[t_1,...,t_{\ell}]$ with $|t_i|=2$,
where $t_i=pr_i^*(c_1)$ is the $1$-st  Chern class
induced from
\[ T=S^1\times ...\times S^1\stackrel{pr_i}{\to}S^1\subset U(1).\]
for the $i$-th projection $pr_i$.
Note that $\ell=rank(G)$ is also the number of the odd degree generators $x_i$ in
$H^*(G;\bZ/p)$.  

It is well
known that $y_i$ are permanent cycles and 
that there is a regular sequence ([Tod2],[Mi-Ni])
$(\bar b_1,...,\bar b_{\ell})$ in $H^*(BT)/(p)$ such that $d_{|x_i|+1}(x_i)=\bar b_i$.  Thus we get
\[ E_{\infty}^{*,*'}\cong grH^*(G/T;\bZ/p)\cong P(y)/p\otimes 
S(t)/(\bar b_1,...,\bar b_{\ell}).\]

 Moreover we know that $G/T$ is a manifold 
such that $H^*(G/T)$ is torsion free, and 
\[(2.2)\quad H^*(G/T)_{(p)}\cong \bZ_{(p)}[y_1,..,y_k]\otimes S(t)/(f_1,...,f_k,b_1,...,b_{\ell})\]
where $b_i=\bar b_i\ mod(p)$ and $f_i=y_i^{p^{r_i}}\ mod(t_1,...,t_{\ell}).$

Let $BP^*(-)$ be the Brown-Peterson theory
with the coefficients ring $BP^*\cong \bZ_{(p)}[v_1,v_2,...],$
$|v_i|=-2(p^i-1)$ ([Ha], [Ra]).
Since $H^*(G/T)$ is torsion free,  the Atiyah-Hirzebruch 
spectral sequence collapses.  Hence  we also know 
\[(2.3)\quad BP^*(G/T)\cong BP^*[y_1,...,y_k]\otimes S(t)/(\tilde f_1,...,\tilde f_k,
\tilde b_1,...,\tilde b_{\ell})\]
where $\tilde b_i=b_i\ mod(BP^{<0})$ and $\tilde f_i=f_i\ mod(BP^{<0}).$

 Let $G_k$ be the split reductive algebraic  group corresponding to $G$, and $T_k$ be the split maximal
torus corresponding to $T$.  Let $B_k$ be the Borel subgroup
with $T_k\subset B_k$.   Note that $G_k/B_k$ is cellular, and
$CH^*(G_k/T_k)\cong CH^*(G_k/B_k)$.
Hence we have   
\[ CH^*(G_k/B_k)\cong H^*(G/T)\quad  and 
\quad  CH^*(BB_k)\cong H^*(BT).
\]

Let $\Omega^*(-)$ be the $BP$-version of the algebraic cobordism ([Le-Mo1,2], [Ya2,4]) 
\[ \Omega^*(X)=MGL^{2*,*}(X)_{(p)}\otimes_{MU_{(p)}^*}BP^*,
\quad \Omega^*(X)\otimes _{BP^*}\bZ_{(p)}\cong CH^*(X)_{(p)}\] where $MGL^{*,*'}(X)$ is the algebraic cobordism theory
defined by Voevodsky with $MGL^{2*,*}(pt.)\cong 
MU^*$ the complex cobordism ring.
There is a natural (realization) map $\Omega^*(X)\to
BP^*(X(\bC))$.  In particular, we have 
$\Omega^*(G_k/B_k)\cong BP^*(G/T).$
Let $I_n=(p,v_1,...,v_{n-1})$ and $I_{\infty}=(p,v_1,...)$
be the (prime invariant) ideals in $BP^*$.  We also note
\[ \Omega^*(G_k/B_k)/I_{\infty}\cong
BP^*(G/T)/I_{\infty}\cong H^*(G/T)/p.\]

\section{The Brown-Peterson theory $BP^*(G/T)$}

Recall that $k(n)^*(X)$ is the connected Morava K-theory
with the coefficients ring $k(n)^*\cong \bZ/p[v_n]$
and $\rho : k(n)^*(X)\to H^*(X;\bZ/p)$ is the natural
(Thom) map.  Recall that there is an exact sequence
(Sullivan exact sequence [Ra], [Ya2])
\[...\to k(n)^{*+2(p^n-1)}(X)\stackrel{v_n}{\to} k(n)^*(X)\stackrel{\rho}{\to}
H^*(X;\bZ/p)\stackrel{\delta}{\to}...\]
such that $\rho\cdot \delta(x)=Q_n(x)$.
Here the Milnor $Q_i$ operation
\[ Q_i: H^*(X;\bZ/p)\to H^{*+2p^i-1}(X;\bZ/p)\]
is defined by $Q_0=\beta$ and $Q_{i+1}=[P^{p^i}Q_i,Q_iP^{p^i}]$
for the Bockstein operation $\beta$ and the reduced power operation $P^j$.

We consider  the Serre spectral sequence  
\[ E_2^{*,*'} \cong H^*(B;H^*(F;\bZ/p))
\Longrightarrow  H^*(E;\bZ/p). \]
induced from the fibering $F\stackrel{i}{\to} E\stackrel{\pi}{\to}B$ with $H^*(B)\cong H^{even}(B)$.
\begin{lemma} (Lemma 4.3 in [Ya1])
In the spectral sequence $E_r^{*,*'}$ above,
suppose that  there is $x\in H^*(F;\bZ/p)$ such that
\[ (*)\quad  y=Q_n(x)\not=0 \quad  and \quad
b= d_{|x|+1}(x)\not = 0\in E_{|x|+1}^{*,0}.\]
Moreover suppose that  $E_{|x|+1}^{0,|x|}
\cong \bZ/p\{x\}\cong \bZ/p$.
%if $ Q_n(\tilde  x)=y,$ then
%\[ (**)\quad d_r(\tilde x)\not =0\ \  for \ r<|x|+1,\quad   %or\quad   
%d_{|x|+1}(\tilde x)=b. \]
Then there are $y'\in k(n)^*(E)$ and $b'\in k(n)^*(B)$ such that
$i^*(y')=y$, $\rho(b')=b$ and that
\[(**)\quad v_ny'=\lambda \pi^*(b')\quad  in \ \  k(n)^*(E),
\ \  for \ \lambda\not =0\in \bZ/p.\]
Conversely if $(**)$ holds in $k(n)^*(E)$ for $y=i^*(y')\not =0$ and $b=\rho(b')\not =0$, then there is $x\in H^*(F;\bZ/p)$ such that $(*)$ holds.
\end{lemma}
\begin{proof}
Let $B'=BT^{|b|-1}$ be the $|b|-1$ dimensional skeleton of $BT$,
and $E'=\pi^{-1}(B')$.  Consider the Serre spectral sequence
\[ E_2^{*,*'}\cong H^*(B';H^*(F;\bZ/p)) \Longrightarrow H^*(E';\bZ/p).\]
Since $d_r(x)=b=0\in H^*(B';\bZ/p)$, there is $x'\in H^*(E';\bZ/p)$ such that $ i^*(x')=x$. Let $Q_n(x')=y'$ so that
$i^*y'=y$. Then $y'$ can be identified as $\delta x'\in k(n)^*(E')$, and  $v_ny'=0$
in $k(n)^*(E')$.

On the other hand, let $B''=B^{|b|-1}\cup e_b$ and $E''=\pi^{-1}B''$ where $e_b$ is the normal cell
representing $b$.
Then $d_rx=b\not =0\in H^*(B'';\bZ/p)$.
By the supposition in this lemma,
there does not exist $x''\in H^*(E'';\bZ/p)$ such that $i^*(Q_nx'')=y$,
that is, for each $y''\in H^*(E'';\bZ/p)$ with $\pi^*y''=y$,
 we see $v_ny''\not =0\in k(n)^*(E'')$.

For $j:E'\subset E''$, we can take an element $y''$ with $j^*(y'')=y'$
by the following reason.
Consider the long exact sequence
\[ ...\to H^*(E'';\bZ/p)\stackrel{j^*}{\to}H^*(E';\bZ/p)
\stackrel{\delta}{\to}H^*(E''/E';\bZ/p)\to ...\]
Here note $H^{|b|}(E''/E';\bZ/p)\cong \bZ/p\{b\}$ and
$\delta(x')=b$.  So we see
\[ \delta(y')=\delta(Q_n(x'))=Q_n(b)=0,\]
since $b\in H^*(B)$. Hence  $y'\in Im(j^*)$.

 Since $v_ny'=v_nj^*(y'')=0\in k(n)^*(E')$ but $v_ny''\not =0\in k(n)^*(E'')$, by dimensional reason, $v_ny''=\lambda b$ for
$\lambda\not =0\in \bZ/p$.

Conversely, suppose that  $v_ny'=\pi^*(b')\not =0$ in $k(n)^*(E)$.
Then $v_ny'=0$ in $k(n)^*(E')$ and there is $\tilde 
x\in H^*(E';\bZ/p)$
with $Q_n\tilde x=y'$. Then for $i^*(y') =y$ and $i^*(\tilde x)=x$,
we see $Q_n(x)=y$. 
But $\tilde x$  does not exist in $H^*(E'';\bZ/p)$.
Hence $d_{|x|+1}(x)=\lambda b$ for $\lambda\not =0\in \bZ/p$, by dimensional reason.
\end{proof}

{\bf Remark.} (Remark 4.8 in [Ya1])
The above lemma also holds letting
$k(0)^*(X)=H^*(X;\bZ)$ and $v_0=p$. 
This fact is well known (Lemma 2.1 in [Tod2]).

%\begin{lemma}  
%Let $F\to E\stackrel{p}{\to} B$ be a fibering, and
%$E(k)=p^{-1}B^k\stackrel{j}{\subset}E$ for $k>0$. 
%Let $E_r^{*,*'}$ (resp.  $E_r(k)^{*,*'}$) be the induced spectral %sequence converging to $H^*(E;\bZ/p)$ (resp. $H^*(E(k);%\bZ/p)$).  Then the induced map
%$ j^*: E_r^{*,*'}\to E(k)_r^{*,*'}$ is injective for $*\le k$
%and surjective for $*\le k-r$.
%\end{lemma}

\begin{cor} 
In the spectral sequence converging to $H^(G/T;\bZ/p)$,
let $b\not =0$ be the transgression image of $x$, i.e.
$d_{|x|+1}(x)=b$.  Then we have the relation
\[ b=\sum_{i=0} v_iy(i)\quad in \ BP^*(G/T)/\II\]
where $y(i)\in H^*(G/T;\bZ/p)$ with $\pi^*y(i)=Q_ix$.
\end{cor}
\begin{proof}
Since $b=0\in H^*(G/T;\bZ/p)$, we can write 
\[ b=py(0)+v_1y(1)+...\quad in\ BP^*(G/T)/\II.\]
If $Q_i(x)=y(i)'\not =0$, then $b=v_iy(i)'$ and take
$y(i)=y(i)'$.  If $Q_i(x)=0$, then $b=0\ mod (v_i^2)$
in $k(i)^*(G/T)$.   Otherwise $b=v_iy(i)'$ with $y(i)'\not =0$
in $H^*(G/T;\bZ/p)$ by Sullivan exact sequence.
Then $Q_i(x)=y(i)'$ from the converse of the preceding lemma.  This is a contradiction. So let $y(i)=0$ when $Q_i(x)=0$.
\end{proof}

Let $G$ be a $simply$ connected Lie group such that
$H^*(G)$ has $p$-torsion.  Then it is known ([Mi-Tod])
that 
$H^*(G)$
has just (not higher) $p$-torsion  in $H^*(G)_{(p)}$.
It is also known that there is $m\ge 1$ with
\[(*)\quad  P^{p^i}(y_i)=y_{i+1}\quad for \ 1\le i\le m-1,
\ \ and \ \  
P^{p^m}(y_m)=0.\]
(Here suffix $i$ is changed adequately
from  that defined in the preceding section
(2.1). Note $m=1$ for type $(I)$ groups.)
Moreover $|x_1|=3$ and $P^1(x_1)=-x_2$, and $\beta(x_2)=y_1$.
We can also take $x_{i+1}$ such that 
\[ (**)\quad Q_i(x_1)=y_{i},\quad Q_0(x_{i+1})=y_i.\]
Therefore from the preceding corollary, we have
\[ b_1=v_1y(1)+...+v_my(m)\quad in\ BP^*(G/T)/\II\]
with $\pi^*(y(i))=y_i$.  We will study the above equation 
in more details.

Here we recall the Quillen (Landweber-Novikov) operation
([Ha], [Ra]).
For a sequence $\alpha=(a_1,a_2,...)$, $a_i\ge 0$ with
$|\alpha|=\sum _i 2(p^i-1)a_i$,  we have the Quillen operation
$ r_{\alpha}: BP^*(X)\to BP^{*+|\alpha|}(X)$
such that 
\[(1)\quad \rho(r_{\alpha}(x))=\chi P^{\alpha}(\rho(x))\quad for \ \rho:BP^*(X)\to
H^*(X;\bZ/p),\]
where $\chi$ is the anti-automorphism in the Steenrod algebra, 
\[ (2)\qquad r_{\alpha}(xy)=\sum _{\alpha=\alpha'+\alpha''}r_{\alpha'}(x)r_{\alpha''}(y)\qquad  Cartan \ formula,\]
\[ (3)\qquad r_{p^i\Delta_{n-i}}(v_n)=
\begin{cases}v_i\qquad for\ \Delta_i=(0,...,0,\stackrel{i}{1},0,...,0). \quad\\
0\quad mod(\II)\quad otherwise.\qquad 
\end{cases} \]
We also note that $\Omega^*(X)$ has the same operation
$r_{\alpha}$ satisfying (2),(3)  and (1) for $\rho: \Omega^*(X)\to CH^*(X)/p$
and the reduced power operation $P^{\alpha}$ on $CH^*(X)/p=H^{2*,*}(X;\bZ/p)$ defined by Voevodsky.

\begin{lemma}
If $|\alpha|<|v_{i-1}|-|v_{i}|=2(p^i-p^{i-1})$, then
$r_{\alpha}$ acts on $BP^*(X)/(\II,v_i,...)$.
\end{lemma}
\begin{proof}  In this case,
 we have
$r_{\alpha}(v_s)\in \II$ for all $s\ge i$.
\end{proof}

Let $h^*(-)$ be a $mod(p)$ cohomology theory (e.g. $H^*(-;\bZ/p)$,
$k(n)^*(-)$).
The product $G\times G\to G$ induces the map
\[ \mu: G\times G/T\to G/T.\]
Here note $h^*(G\times G/T)\cong h^*(G)\otimes_{h^*}h^*(G/T)$, since $h^*(G/T)$ is
$h^*$-free.  For $x\in h^*(G/T)$, we say
that $x$ is $primitive$ ([Mi-Ni], [Mi-Tod])  if
\[ \mu^*(x)=\pi^*(x)\otimes 1\ +\ 1\otimes x\quad where \
\pi:G\to G/T.\]

It is immediate that $x$ is primitive implies so is 
 $r_{\alpha}(x)$.
Of course $b\in BP^*(BT)$ are primitive but $by_i$ are not,
in general.
We can take $y_1$ as primitive (adding elements
if necessary) in $BP^*(G/T)$.

\begin{lemma} Let $G$ be a simply connected Lie group satisfying $(*)$.
Let $y_1$ be a primitive element in $BP^*(G/T)$, and define $y_{i+1}=r_{p^i\Delta_1}(y_i)$.  Then we have
\[ v_1y_1+v_2y_2+...+v_my_m=b_1\quad mod(\II).\]
\end{lemma}
\begin{proof} 
Note that  $v_ny(n)=b_1$ is of course primitive in $k(n)^*(G/T)$.
We prove $y_n=y(n)\ mod(\II)$.  Let us write
\[ y(n)=y_n+\sum yt\quad with \  y\in P(y),\ t\in S(t),
 \ |t|\ge 2.\]

We will prove $t=0$.
Consider the Atiyah-Hirzebruch spectral sequence
\[ E_2^{*,*'}\cong H^*(G;k(n)^*)\Longrightarrow
        k(n)^*(G).\]
The first non-zero differential is $d_{2p^n-1}(x)=v_nQ_n(x).$
Since $|y|\le |y_n|-2= 2p^n$, we see  that $y$ is $v_n$-torsion free in $k(n)^*(G)$. This means
if $t\not =0$, then
\[  v_n y\otimes t\not =0\quad  in\ k(n)^*(G)\otimes k(n)^*(G/T).\]
Therefore $t=0$ since f $y_n$ and $v_ny(n)$ are primitive.
\end{proof}

Act $r_{\Delta_1}$ on the above equation, we have 
\begin{lemma}  In $BP^*(G/T)/(\II)$, we have
\[py_1+v_1P^1(y_1)+v_2P^1(y_2)+...+v_mP^1(y_m)=
P^1(b_1)=b_2.\]
\end{lemma}

\section{versal flag varieties}

Recall that $\bG_k$ is a  nontrivial $G_k$-torsor.
We can construct a twisted form of $G_k/B_k$ by
\[(\bG_k\times G_k/B_k)/G_k\cong \bG_k/B_k.
\]
We will study  the twisted flag variety $X=\bG_k/B_k$.

Let $P\supset T$ be a parabolic subgroup of $G$.
Petrov, Semenov and Zainoulline developed the theory of decompositions of motives $M(\bG_k/P_k)$.
They  develop the theory of generically split varieties.
We say that  $L$ is splitting field of a variety of $X$ if $M(X|_L)$ is isomorphic to a
direct sum of twisted Tate motives $\bT^{\otimes i}$
and the restriction map $i_L: M(X)\to M(X|_L)$ is isomorphic after tensoring $\bQ$.
A smooth scheme $X$ is said to be generically split over $k$ if its function field $L=k(X)$
is a  splitting field.  Note that (the complete flag) $X=\bG_k/B_k$
is always generically split, i.e., $X|L$ is cellular.

\begin{thm} (Theorem 3.7 in [Pe-Se-Za])
Let $Q_k\subset P_k$ be parabolic subgroups of $G_k$ which are generically split over $k$.
Then there is a decomposition of motive :
$M(\bG/Q_k)\cong M(\bG_k/P_k)\otimes H^*(P/Q)$.
\end{thm}

By extending the arguments by Vishik [Vi] for quadrics to that for  flag varieties, Petrov, Semenov and Zainoulline define the 
$J$-invariant of $\bG_k$. Recall the expression in $\S 2$ 
\[(*)\quad H^*(G;\bZ/p)\cong \bZ/p[y_1,...,y_s]/(y_1^{p^{r_1}},...,y_s^{p^{r_s}})
\otimes \Lambda(x_1,...,x_{\ell}).
\]
Roughly speaking (for the accurate definition, see
[Pe-Se-Za]),
the $J$-invariant is defined as 
$ J_p(\bG_k)=(j_1,...,j_{s})$
if $j_i$ is the minimal integer such that
\[ y_i^{p^{j_i}}\in Im(res_{CH})\quad mod (y_1,...,y_{i-1}, t_1,...,t_{\ell})\]
for $res_{CH}:CH^*(\bG_k/B_k)\to CH^*(G_k/B_k)$.
Here  we take $|y_1|\le |y_2|\le...$ in $(*)$.  Hence $0\le j_i\le r_i$ and 
$J_p(\bG_k)=(0,...,0)$ if and only if $\bG_k$ split by an extension of the index coprime to $p$.
One of the main results in [Pe-Se-Za] is 
\begin{thm} (Theorem 5.13 in [Pe-Se-Za]
and Theorem 4.3 in [Se-Zh])
Let $\bG_k$ be a $G_k$-torsor over $k$, $X=\bG_k/B_k$
 and $J_p(\bG_k)=(j_1,...,j_{s})$.
Then there is a $p$-localized  motive $R(\bG_k)$ such that
\[M(X)_{(p)}\cong  \oplus _u R(\bG_k)\otimes \bT^{\otimes u}.\]
Here $\bT^{\otimes u}$ are Tate motives with 
$CH^*(\oplus_u \bT^{\otimes u})/p\cong  P'\otimes S(t)/(b)$ 
where \[P'(y)=\bZ/p[y_1^{p^{j_1}},...,y_{s}^{p^{j_{s}}}]/(y_1^{p^{r_1}},...,y_{s}^{p^{r_{s}}}
)\subset P(y)/p,\]
\[\quad S(t)/(b)=S(t)/(b_1,...,b_{\ell}).\]
The $mod(p)$ Chow group of $\bar R(\bG_k)=R(\bG_k)\otimes \bar k$  is  given by 
\[ CH^*(\bar R(\bG_k))/p\cong \bZ/p[y_1,...,y_s]/(y_1^{p^{j_1}},...,y_{s}^{p^{j_{s}}}).
 \]
Hence we have $ CH^*(\bar X)/p\cong  CH^*( \bar R(\bG_k))\otimes
P'(y)\otimes S(t)/(b) $ and \\
$ CH^*(X)/p\cong 
CH^*( R(\bG_k))\otimes P'(y)\otimes S(t)/(b).$
\end{thm}

Let $P_k$ be special (namely, any extension is split,
e.g. $B_k$). 
Let us consider an embedding of $G_k$ into the general linear group $GL_N$ for some $N$.  This makes $GL_N$ a $G_k$-torsor over the quotient variety $S=GL_N/G_k$.
We define $F$ to be the function field $k(S)$ and  define
the $versal$ $G_k$-$torsor$ $E$ to be the $G_k$-torsor over $F$ given by the generic fiber of $GL_N\to S$. 
(For details, see [Ga-Me-Se], [To2], [Me-Ne-Za], [Ka1].) 
\[\begin{CD}
        E@>>> GL_N\\
       @VVV     @VVV \\
     Spec(k(S))  @>>> S=GL_N/G_k
\end{CD}\]
The corresponding flag variety $E/P_k$ is called $generically$ $twisted$ or $versal$ flag
variety, which is considered as the most complicated twisted
flag variety (for given $G_k,P_k$). It is known that the Chow ring
$CH^*(E/P_k)$ is not dependent to the choice
of  generic $G_k$-torsors $E$ (Remark 2.3 in [Ka1]).

Karpenko and Merkurjev proved the following result
(for $CH^*(X)$) for a versal (generically twisted) flag variety.  
\begin{thm}
(Karpenko Lemma 2.1 in [Ka1])
Let $h^*(X)$ be an oriented cohomology theory
(e.g., $CH^*(X)$, $\Omega^*(X)$).
Let $P_k$ be a parabolic subgroup of $G_k$ and $\bG_k/P_k$ be a versal flag variety.
Then the natural map
$h^*(BP_k)\to h^*(\bG_k/P_k)$ is surjective.
\end{thm}
\begin{cor}
The Chow ring $CH^*(\bG_k/B_k)$
is generated by elements $t_i$ in $S(t)$. In particular,
each $x\in CH^*(G_k/B_k)$, the element
$p^sx $ is represented by elements in $S(t)$
for a sufficient large  $s$.
\end{cor}
\begin{proof}  For some extension $F/k$ of order $ap^s$
and $a$ is coprime $p$, the $G_k$-torsor $\bG_k$ splits.
Hence $p^sy^i\in Im(res_{CH}:CH^*(\bG_k/B_k)\to CH^*(G_k/B_k))$,
which is written by elements in $S(t)$ by the above
Karpenko theorem.
\end{proof}
\begin{cor} If a $G_k$-torsor  $\bG_k$ is 
 versal, then  $J(\bG_k)=(r_1,...,r_s)$, i.e. $r_i=j_i$.
\end{cor}
\begin{proof}
If $j_i<r_i$, then $0\not =y_i^{p^{j_i}}\in res(CH^*(X)\to
CH^*(G_k/B_k))$, which is in the image from $S(t)$
by the preceding theorem.
This contradicts to $CH^*(G_k/T_k;\bZ/p)\cong
P(y)/p\otimes S(t)/(b)$ and $0\not =y_i^{p^{j_i}}\in P(y)/p$.
\end{proof}

Here we recall the (original) Rost motive $R_a$
(we write it by $=R_n$)
defined from a nonzero pure symbol $a$ in $K_{n+1}^M(k)/p$.
When $J(\bG_k)=(1)$ (and $G$ is simply connected), we know $R(\bG_k)\cong R_2$
from [Pe-Se-Za].  We write 
$\bar R_n=R_n\otimes \bar k$.  The Rost motive  $R_n$ is defined as a non-split motive but split over a field of degree
$ap$ with $(a,p)=1$, and  for $|y|=2b_n=2(p^n-1)/(p-1)$
\[ CH^*(\bar R_n)\cong \bZ[y]/(y^p),\quad \Omega^*(\bar
R_n)\cong BP^*[y]/(y^p). \]
\begin{thm} ([Vi-Ya], [Ya4], [Me-Su])
Let $R_n$ be the (original) Rost motive
defined by Rost and Voevodsky ([Ro1,2],[Vo2,3]). 
Then the restriction 
$ res_{\Omega}:\Omega^*(R_n)\to \Omega^*(\bar R_n)$
is injective.  Recall $I_n=(p,...,v_{n-1})\subset BP^*$.
The restriction image $Im(res_{\Omega})$ is isomorphic to
\[ BP^*\{1\}\oplus I_n[y]^{+}/(y^p)\]
 \[ \cong  BP^*\{1,v_jy^i\ |\ 0\le j\le n-1,
1\le i\le p-1\}\subset BP^*[y]/(y^p).
\]
Hence writing $v_jy^i=c_{j}(y^i)$, $|c_j(y^i)|=2ib_n-2(p^j-1)$,
we have 
\[CH^*(R_n)/p\cong \bZ/p\{1,c_j(y^i)\ |\ 0\le j\le n-1,\ 
1\le i\le p-1\}.\]
%Hence writing $c_j(y^i)=v_jy^i$, we have
%\[ CH^*(R_n)/p\cong \bZ/p\{1,c_0(y),c_1(y),...,c_{n-2}(y^{p-1}), %c_{n-1}(y^{p-1})\}.\]
\end{thm}
{\bf Example.}  In particular,   we have  isomorphisms
%\[ CH^*(R(\bG_k))/p\cong \bZ/p\{1,c_{0}(y),c_{1}(y),...,c_{0}(y^{p-1}), c_{1}(y^{p-1})\}.\]
\[ CH^*(R_1)/p\cong \bZ/p\{1,c_0(y),...,c_0(y^{p-1})\},\]
\[ CH^*(R_2)/p\cong \bZ/p\{1,c_0(y),c_1(y),...,c_0(y^{p-1}),c_1(y^{p-1})
\}.\]

\section{torsion index}

Let $dim(G/T)=2d$.  Then the torsion index is defined as
\[ t(G)=|H^{2d}(G/T;\bZ)/H^{2d}(BT;\bZ)|.\]
Let $n(\bG_k)$ be the greatest common divisor of the degrees of all finite field extension $k'$ of $k$ such that $\bG_k$ 
becomes trivial over $k'$.  Then by Grothendieck [Gr], it is known that $n(\bG_k)$ divides $t(G)$.  Moreover, there is a $G_k$-torsor $\bG_F$ over some extension field $F$ of $k$ such that $n(\bG_F)=t(G)$   (in fact, this holds for each  versal $G_k$-torsor [To2], [Me-Ne-Za], [Ka1]).
Note that $t(G_1\times G_2)=t(G_1)\cdot t(G_2)$.
It is well known that  if $H^*(G)$ has a $p$-torsion, then
$p$ divides the torsion index $t(G)$. Torsion index for
simply connected compact Lie groups are completely determined by Totaro [To1,2].
%Hence we also see that there is   $k$ such that there is non %trivial $G_k$-torsor $\bG_k$. 
For example,  $t(E_8)=2^63^25.$

Hereafter this paper, we assume that 
$\bG_k$ be  a $versal$ $G_k$-torsor
and $X=\bG_k/B_k$ is  the $versal$ flag variety.  Recall that 
\[ grH^*(G/T;\bZ/p)\cong P(y)/p\otimes S(t)/(b)\]
\[ where \ \   S(t)/(b)=S(t)/(b_1,...,b_{\ell}),
\ \ P(y)/p\cong 
\bZ/p[y_1,...,y_s]/(y_1^{p^{r_1}},...,y_s^{p^{r_s}}).\]
Recall Theorem 4.2 and Corollary 4.5,  and we see $J(\bG_k)=(r_1,...,r_s)$,
i.e., $y_i^{p^{r_i}-1}\not \in S(t)$.

Giving the filtration on $S(t)$ by $b_i$, we have the isomorphism
\[ gr S(t)/p\cong \bZ/p[b_1,...,b_{\ell}]\otimes S(t)/(b_1,...,b_{\ell}).\]
Let us write  for $N>0$
\[ A_N=\bZ/p\{b_{i_1}...b_{i_k}|\ |b_{i_1}|+...+|b_{i_k}|\le N\}\subset
    grS(t).\]
Of course $H^*(G/T)=0$ for $*>2d=dim(G/T)$, we have a map
\[ gr S(t)/p\to A_{2d}\otimes S (t)/(b)\to grCH^*(X)/p.\]

\begin{lemma}  The composition map is  a surjection
\[ 
A_{2d}\to CH^*(X)/p \stackrel{pr}{\twoheadrightarrow}
 CH^*(R(\bG_k))/p.\]
\end{lemma} 
\begin{proof}
Recall the decomposition
$M(X)_{(p)}\cong \oplus_i R(\bG_k)\otimes \bT^{s_i}$.  Since
the restriction map $res_{CH} : CH^*(\bT^{s_i})/p\to
CH^*(\bar \bT^{s_i})/p$ is an isomorphism, we have
\[ CH^*(\oplus_i\bT^{s_i})/p\cong CH^*(\oplus _i\bar \bT^{s_i})/p\]
\[ \cong CH^*(G_k/T_k)/(p,P(y)^+)\cong S(t)/(p,b).\]
Thus 
we can write
$CH^*(\bT^{s_i})\cong \bZ_{(p)}\{u_i\}$ for some
$u_i\not =0\in S(t)/(p,b)$.  Hence $CH^*(X)/p$ is generated 
by elements which are product $b\cdot u$ in $CH^*(X)/p$ for $b\in CH^*(R(\bG_k))\subset CH^*(X)/p$ and $u\in S(t)/(p,b)$. Note $bu\not =0$ if $b\not =0$.

 On the other hand, since  $CH^*(X)$ is versal and  generated by images from $S(t)$, which is generated by $b'u$ for
$b'\in Im(A_d\to CH^*(X)/p)$.
When $s_i\not =0 $  (i.e., $|u|\ge 2$), we see  $pr(b'u)=0$ for the projection
$pr:CH^*(X)/p\to CH^*(\bG_k))/p$.  Hence we have the
lemma.
\end{proof}

From the arguments in the proof of preceding lemma,  we have
\begin{cor}
If $b\in Ker(pr)$, then we can write
\[ b=\sum b'u'\quad with \ b'\in A_{2d},\ 0\not = u'\in S(t)/(p,b),\ \ |u'|>0.\]
\end{cor}
\begin{cor}
If  $b_i\not=0$ in $CH^*(X)/p$,  then  so in $CH^*(R(\bG_k))/p$.
\end{cor}
\begin{proof}
Let $pr(b_i)=0$.  Then
 $b_i=\sum b'u'$ for $|u'|>0$, and hence $b'\in Ideal(b_1,...,b_{i-1})$.
This contradict to that $(b_1,...,b_{\ell})$ is regular.
\end{proof}

Let us write 
\[  y_{top}=\Pi_{i=1}^s y_i^{p^{r_i}-1}\quad (reps.\ t_{top})\]
the generator of the highest  degree 
in $P(y)$ (resp. $S(t)/(b)$) so that $f=y_{top}t_{top}$
is the fundamental class in $H^{2d}(G/T)$.
\begin{lemma}  The following map is surjective
\[ A_N\twoheadrightarrow CH^*(R(\bG_k))/p\quad where\ 
N=|y_{top}|.\]
\end{lemma}
\begin{proof}  In the preceding lemma,  $A_{N}\otimes u$ for $|u|>0$ maps zero in $CH^*(R(\bG_k))/p$.  Since each element
in $S(t)$ is written by an element in $A_N\otimes S(t)/(b)$,
we have the corollary.
\end{proof}

{\bf  Remark.}
In $\S 7$ in [Pe-Se], Petrov and Semenov  show
\[CH^*(BB_k)/p\cong CH_{G_k}(\bG_k/B_k)/p
\cong \oplus CH_{G_k}^*(R_{p,G_k}(\bG_k))/p\]
where $CH_{G_k}(-)$ is the $G_k$-equivariant Chow ring
and $R_{p,G_k}(\bG_k)$ is the $G_k$-equivariant generalized 
Rost motive.  Hence we have
\[CH^*_{G_k}(R_{p,G_k}(\bG_k))/p\cong A_{\infty}=\bZ/p[b_1,...,b_{\ell}].  \]

Now we consider the torsion index.
\begin{lemma}
Let $\tilde b=b_{i_1}... b_{i_k}$ in
$S(t)$  such that
in $H^*(G/T)_{(p)}$
\[ \tilde b=p^s(y_{top}+\sum yt),\quad
|t|>0\]
for some $y\in P(y)$  and $t\in S(t)$.
Then the torsion index $t(G)_{(p)}\le p^s$. 
%and $t(G)_{(p)}$
%is the smallest $p^s$ satisfied above equation.
\end{lemma}
\begin{proof}
Suppose $p^s<t(G)_{(p)}$.  We can assume $t(G)=p^{s+1}
$ multiplying $p^i$ if necessary.  Since $tt_{top}=0\in
S(t)/(b)$, we see
\[tt_{top}\in Ideal(b_1,...,b_{\ell}) \subset Ideal(p).\]
Therefore $p^s\sum ytt_{top}\in Ideal(p^{s+1})$. So it is in
$S(t)$, by Karpenko's theorem.  
Hence $p^sy_{top}t_{top}\in S(t)$.  So $t(G)\le p^{s}$
and this is a contradiction,
%
%Suppose $t(G)_{(p)}=p^s$.  Then $p^sy_{top}t_{top}\in S(t)$.
%Recall the map $S(t)\to A(b)\otimes S(t)/(b)\to CH^*(X)/p$.
%Hence we can write $p^sy_{top}t_{top}=\sum b(A)t$ with
%$b(A)\in A(b)$ and $t\in S(t)/(b)$.  If $|t|>0$, then 
\end{proof}

\begin{cor}
In the preceding lemma, assume $p^s=t(G)_{(p)}$.
Then for each subset $(i_1',...,i_{k'}')\subset (i_1,...,i_k)$, the element $b_{i_1'}'...b_{i_{k'}}'\not =0 \in CH^*(X)/p$.
\end{cor}
\begin{proof} Let us write
$I'=(i_1',...,i_{k'}')\subset I=(i_1,...,i_k)$, $I'\cup I''=I$, and 
$b_I=b_{i_1}...b_{i_k}$.
It is immediate $b_{I'}\not =0\in CH^*(X)/p$ since
$b_I=b_{I'}b_{I''}\not =0 \in CH^*(X)/p$.%
\end{proof}
 
So when $t(G)_{(p)}$ is big enough and there is $\tilde b$ in the preceding lemma,  we can
find many non zero  elements in $CH^*(X)/p$ whose 
restriction images are zero in $CH^*(\bar X)/p$.

\section{ The groups  $GL(n)$, 
$Sp(n)$ and $PU(p)$.}

Results in this section are known.  However we write down them since results and arguments are used other sections.
We consider the Lie group $G=U(\ell)$ at first.
Note that its cohomology  has no torsion.  Recall that
\[H^*(U(\ell))\cong \Lambda(x_1,...,x_{\ell})\quad
with\ |x_i|=2i-1.\]
So $P(y)/p\cong \bZ/p$, and
$ CH^*(R(\bG_k)/p\cong CH^*(\bar R(\bG_k))/p\cong 
 \bZ/p,$
that is, there is no twisted form of $G_k/B_k$.
Moreover
$CH^*(X)/p\cong S(t)/(p,b_1,...,b_{\ell})$
for $d_{|x_i|+1}(x_i)=b_i$. It is well known that
we can take $b_i=c_i$ the $i$-th elementary symmetric 
function on $S(t)\cong \bZ[t_1,...,t_{\ell}]$
\begin{prop}
Let $G=U(\ell)$ (i.e., $G_k=GL_{\ell}$) and $p$ is a prime number.
Let $X=G_k/B_k$.  Then
\[ CH^*(X)/p\cong S(t)/(p,c_1,...,c_{\ell}) \]
where $c_i$ is the Chern class in $H^*(BT)\cong S(t)$
by the map $T\subset U(\ell)$.
\end{prop}
\begin{proof}
We consider the fibering $G/T\to BT\to BG$.
The composition of the induced maps
$ H^*(BG)\to H^*(BT)\to H^*(G/T)$
is zero.  The first map induces the isomorphism
\[ H^*(BG)\cong H^*(BT)^{W_G(T)}\cong \bZ[c_1,...,c_{\ell}]
\]
Thus $(b_1,...,b_{\ell})\supset (c_1,...,c_{\ell})$.  By dimensional reason, we have the proposition.
\end{proof}

Next consider in the case $G'=Sp(\ell)$ and
recall that
\[H^*(Sp(\ell))\cong \Lambda(x_1',...,x_{\ell}')\quad
with\ |x_i'|=4i-1.\]
So $P(y)'/p\cong \bZ/p$, and
there is no twisted form of $G_k'/B_k$.
Moreover we have 
 $d_{|x_i'|+1i}(x_i')=p_i$  the Pontryagin 
 class.
 Hence we have
\begin{prop}
Let $G'=Sp(\ell)$ and 
and $X'=G_k'/B_k$.  Then
for each prime number $p$,  we have
\[ CH^*(X')/p\cong S(t)/(p,p_1,...,p_{\ell}).\]
In particular, when $p=2$,  we have $ CH^*(X')/2\cong S(t)/(2,c_1^2,...,c_{\ell}^2).$
\end{prop}

Now we consider in the case $(G,p)=(PU(p),p)$,
which has $p$-torsion in cohomology,
but it is not simply connected.
 Its mod $p$ cohomology is
\[ H^*(G;\bZ/p)\cong \bZ/p[y]/(y^p)\otimes
\Lambda(x_1,...,x_{p-1})\quad |y|=2,\ |x_i|=2i-1.\]
So $P(y)/p\cong \bZ/p[y]/(y^p)$ with $|y|=2$.
This fact is given by the cofibering $U(p)\to PU(p)\to BS^1$
and the induced spectral sequence
\[ E_2^{*,*'}\cong H^*(BS^1;H^{*'}(U(p);\bZ/p))
\Longrightarrow H^*(PU(p);\bZ/p).\]
Here we use that $H^*(BS^1;\bZ/p)\cong \bZ/p[y]$
and $d_{2p}x_{p}=y^p$.

Since $G$ is not simply connected, $G$ is not of type $I$
while $P(y)$ is generated by only one $y$. (However
$CH^*(X)/p$ is quite  resemble to that of 
type $(I)$. Compare Theorem 6.5 and Theorem 9.4
below.)  

We consider the map $U(p-1)\to U(p)\to PU(p)$ where
the maximal tori of $U(p-1)$ and $PU(p)$ 
are isomorphic, i.e., $T_{U(p-1)}\cong T_{PU(p)}$.
By using the map $U(p-1)\to PU(p)$, we know
$d_{2i}(x_i)=c_i$. Hence we have
\[ grH^*(G/T;\bZ/p)\cong \bZ/p[y]/(y^p)\otimes
S(t)/(c_1,...,c_{p-1}).\]

\begin{lemma}
Let  $X$ split over a field $k'$ over $k$ of index
$p^t\cdot a$ for some $a$ coprime to $p$.
Then for all $y\in CH^*(\bar X)$, we see
$p^ty\in Im(res_{CH})$.
\end{lemma}
\begin{proof} 
Using the fact that $res\otimes \bQ$ is isomorphic, there is $s$ such that $p^sy=res(x)$ for some $x\in CH^*(X)$.
Then
\[ p^s res\cdot tr(y)=res\cdot tr \cdot res(x)
  =res(ap^tx)=ap^{s+t}(y).\]
Since $CH^*(\bar X)$ is torsion free, we have
  $res\cdot tr(a^{-1}y)=p^ty$.
\end{proof}

\begin{lemma} We have $py^i=c_i\in H^*(G/T)_{(p)}$.
\end{lemma}
\begin{proof}
By induction on $i$, we will prove $py^i=c_i$.
It is known from [Pe-Se-Za] that
$R(\bG_k)\cong R_1$ (note $\bG_k$ is versal).  
From the preceding lemma, $py\in Im(res_{CH})$.  By
Karpenko's theorem, $py^i$ is represented by elements in $CH^*(BT)$.   
Since $py^i\in Ideal(c_1,...,c_i)$, we can write
\[py^i=\sum_{j<i}c_jt(j)+\lambda c_i \quad
for \  t(j)\in S(t),\ \lambda \in \bZ.\]
If $\lambda =0\in \bZ/p$, we see 
$py^i=\sum py^jt(j)$ by inductive assumption,  and this is a contradiction,
since $ CH^*(\bar X)$ is $p$-torsion free.
\end{proof}
%\begin{proof}
%Since $H^*(G/T)$ is torsion free, $py^i\not =0$ in
%$H^*(G/T)$.   If $py^i\not\in S(t)$,
%then we can inductively write
%\[ py^i=\sum _{0<s<i}y^st(s)\quad for \ some \ s,\  t(s)\not =0\in
%S(t)/(p,c_1,...,c_{i-1}).\]
%But it contradicts to $H^*(G/T)/p\cong P(y)\otimes S(t)/(p,b)$
%for  $b_i=c_i$.
%\eIt is known from [Pe-Se-Za] that
%$R(\bG_k)\cong R_1$.  \end{proof}

\begin{thm} Let $G=PU(p)$ and $X=\bG_k/B_k$.  Then there are isomorphisms
\[ CH^*(R(\bG_k))/p \cong CH^*(R_1)/p\cong \bZ/p\{1,c_1,...,c_{p-1}\},  \]
\[ CH^*(X)/p\cong S(t)/(p,c_ic_j|1\le i,j\le p-1).\]
\end{thm}
\begin{proof}
From [Pe-Se-Za], recall 
$R(\bG_k)\cong R_1$.  
Hence the second isomorphism follows from $py^i=c_i$ and  (Example of) Theorem 4.6,
\[ CH^*(R_1)/p\cong \bZ/p\{1,py,...,py^{p-1}\}.\]
From  the main theorem of [Pe-Se-Za], we 
have the additive isomorphism
\[ CH^*(X)/p\cong \bZ/p\{1,py,...,py^{p-1}\}\otimes
S(t)/(b)\quad where\ b_i=c_i.\]
Note $c_ic_j=p^2y^{i+j}=pc_{i+j}$ in $\Omega^*(\bar X)$.
Since $res_{\Omega}:\Omega^*(X)\to \Omega^*(\bar X)$
is injective,  we see $c_ic_j=0\in CH^*(X)/p$.

Of course we have an additive isomorphism
\[ S(t)/(p,c_ic_j)\cong \bZ/p\{1,c_1,...,c_{p-1}\}\otimes
S(t)/(c_1,...,c_{p-1}).\] 
Moreover we have a surjective   ring map $S(t)/(p,c_ic_j)\twoheadrightarrow CH^*(X)/p$.  From the additive isomorphism, its kernel is zero,
which induces the ring isomorphism of the theorem.
\end{proof}
Since $CH^*(X)$ is torsion free, we also get the above theorem considering the restriction map $CH^*(X)\to
CH^*(\bar X)$.

We note here the following lemma for a (general)
 split algebraic group
$G_k$ and a $G_k$-torsor $\bG_k$.
\begin{lemma} The composition of the following maps is zero
for $*>0$
 \[ CH^*(BG_k)/p\to CH^*(BB_k)/p\to CH^*(\bG_k/B_k)/p.\]
\end{lemma}
\begin{proof} Take $U$ (e.g.,  $GL_N$ for a large $N$)
such that $U/G_k$ approximates the classifying space
$BG_k$ [To3].  Namely, we can take $\bG_k=f^{*}U$ for the classifying map
$f: \bG_k/G_k \to U/G_k$.
Hence we have  the following commutative diagram
\[  \begin{CD}
   \bG_k/B_k @>>> U/B_k\\
       @VVV      @VVV\\
    Spec(k)\cong \bG_k/G_k @>>> U/G_k
\end{CD}\]  
where $U/B_k$ (resp. $U/G_k$)
approximates $BB_k$ (resp. $BG_k$).
Since $CH^*(Spec(k))/p=0$ for $*>0$, we have the lemma.
\end{proof}

\section{The orthogonal group $SO(m)$ and $p=2$}

 We  consider the
orthogonal groups $G=SO(m)$ and $p=2$
in this section.
The mod $2$-cohomology is written as ( see for example [Mi-Tod], [Ni])
\[ grH^*(SO(m);\bZ/2)\cong \Lambda(x_1,x_2,...,x_{m-1}) \]
where $|x_i|=i$, and the multiplications are given by $x_s^2=x_{2s}$.
We write $y_{2(odd)}=x_{odd}^2$.     Hence we can write
\[H^*(SO(m);\bZ/2)\cong P(y)\otimes \Lambda(x_1,x_3,...x_{\bar m}),\]
\[ with \ \ P(y)=\otimes_{i=0}^s\bZ/2[y_{4i+2}]/(y_{4i+2}^{2^{r_i}}),\quad grP(y)\cong \Lambda(x_2,x_4,...x_{m'})\]
for adequate integers $\bar m, m', s, r_i$. 
%where $2\le 4i+2\le m-1$, and $s(i)$ is the smallest number %such that $m\le 2^{s(i)}(4i+2)$, $t$ is the largest number
%such that $4t+1\le m-1$, and where $\bar m=m-1$ (resp.%$\bar m=m-2$) if $m$ is even (resp. odd)  and
%where $m'=m-2$ (resp. $m'=m-1$) if $m$ is even (resp. odd).
%The index means its degree, namely $|y_j|=j, |x_k|=k$.
For ease of argument,  at first, we only consider in the case
$m=2\ell+1$  so that
\[ H^*(G;\bZ/2)\cong P(y)\otimes \Lambda(x_1,x_3,...,x_{2\ell-1}) \]
\[ grP(y)/2\cong \Lambda(y_2,...,y_{2\ell}), \quad 
letting\ y_{2i}=x_{2i}\ \ (hence \ y_{4i}=y_{2i}^2).\]
(Note that  the suffix means its degree and it is used differently
from other sections.)

The Steenrod operation is given as 
$Sq^k(x_i)= {i\choose k}(x_{i+k}).$
The $Q_i$-operations are given by Nishimoto [Ni]
\[Q_nx_{2i-1}=y_{2i-2^{n+1}-2},\qquad Q_ny_{2i}=0.\]

Considering the map
$U(\ell)\to SO(2\ell)\to SO(2\ell+1),$
we see that $b_i=c_i$ $mod(2)$ for the transgression 
$d_{2i}(x_{2i-1})=b_i$ and $c_i$ which is the  $i$-th elementary symmetric function
on $S(t)$, from Proposition 6.1
in the preceding section.  Moreover we see
$Q_0(x_{2i-1})=y_{2i}$ in $H^*(G;\bZ/2)$.
From Lemma 3.1 or Corollary 3.2, we have
\[ 2y_{2i}=c_i\ mod(4)\quad in\ H^*(G/T).\]
Indeed, the cohomology  $H^*(G/T)$ is computed 
completely by
Toda-Watanabe [Tod-Wa]
\begin{thm} ([Tod-Wa]) 
There are $y_{2i}\in H^*(G/T)$ for $1\le i\le \ell$
such that $\pi^*(y_{2i})=y_{2i}$ for $\pi: G\to G/T$, and that 
we  have an isomorphism
\[ H^*(G/T)\cong \bZ[t_i,y_{2i}]/(c_i-2y_{2i},J_{2i})\]
where $J_{2i}= 1/4(\sum_{j=0}^{2i}(-1)^jc_jc_{2i-j})=
y_{4i}-\sum_{0<j<2i}(-1)^jy_{2j}y_{4i-2j}$ \\
letting $y_{2j}=0$ for $j>\ell$.
\end{thm}

By using Nishimoto's result for $Q_i$-operation, 
from Corollary 3.2, we have
\begin{cor} In $BP^*(G/T)/\II$, we have  
\[c_i= 2y_{2i}+\sum v_n(y(2i+2^{n+1}-2)) \]
for some $y(j)$ with $\pi^*(y(i))=y_{i}$.
\end{cor}

It is known by Marlin and Merkurjev (see [To2] for details)
that the torsion
index of $SO(2\ell+1)$ (and $SO(2\ell+2)$) is $2^{\ell}$.
Here we give an another proof.
\begin{thm} $t(G)=t(SO(2\ell+1))=2^{\ell}$.
\end{thm}
\begin{proof} We consider in $H^*(G/T)$
\[ c_1...c_{\ell}=(2y_2)(2y_4)...(2y_{2\ell})=
2^{\ell}y_{top}\]
where  $y_{top}=y_2...y_{2\ell}.$
Hence $t(G)\le 2^{\ell}$.

Conversely, let $2^{\ell-1}y_{top}=t$ in $S(t)$.
Then $t$ in the ideal $(c_1,...,c_{\ell})$ in $S(t)$.  So we can write
$t=\sum c_it(i)$.  Then we have 
\[ 2^{\ell-1}y_{top}=2\sum y_{2i}t(i)\]
which implies $2^{\ell-2}y_{top}=\sum y_{2i}t(i)$ since
$H^*(G/T)$ has no torsion.  Continue this argument, we have a relation
$  y_{top}=\sum yt$  with $t\in S(t)$
where  the number of $y_{2s}$in each  monomial in $y$ is less or equal to $\ell-1$, while the number for $y_{top}$ is $\ell$.  This contradicts to $H^*(G/T)/2\cong
P(y)\otimes S(t)/(2,b)$.
\end{proof}

%Let  $W=W_{SO(2\ell+1)}(T)$
%be the Weyl group.
%Then $W\cong S_k^{\pm}$ is generated by permutations and change of signs so that $|S_k^{\pm}|=2^kk!$.
%Hence 
%we have
%\[H^*(BT)^{W}\cong \bZ_{(2)}[p_1,...,p_k]\subset H^*(BT)\cong \bZ_{(2)}[t_1,...,t_k],\ |t_i|=2 \]
%where the Pontryagin class $p_i$ is defined by
%$\Pi_i(1+ t_i^2)=\sum_ip_i$.    Consider the maps
%\[ \eta : T \stackrel{\eta_1}{\subset} U(\ell){\to} 
%    SO(2\ell+1)\stackrel{\eta_2}{\to} U(2\ell+1).\]
%Then $c_{2i}(\eta)=p_i\in CH^*(BT)^{W}$ which is the image
%from $c_{2i}(\eta_2)\in CH^*(BSO(2\ell+1))$.  
%
%On the other hand, $p_i=c_i(\eta_1)^2\ mod (2)$,
%where $c_i(\eta_1)=\sigma_i$ is the elementary symmetric %function in $S(t)$.  
%From Lemma 6.5, the composition of the following maps
%\[ CH^*(BG_k)/2\to CH^*(BB_k)/2\to CH^*(X)/2\]
%is zero for $*>0$,  we get 
%$ c_i(\eta_1)^2=\sigma_i^2=0$ in $CH^*(X)/2$.

Let  $W=W_{SO(2\ell+1)}(T)$
be the Weyl group.
Then $W\cong S_{\ell}^{\pm}$ is generated by permutations and change of signs so that $|S_{\ell}^{\pm}|=2^{\ell}{\ell}!$.
Hence 
we have
\[H^*(BT)^{W}\cong \bZ_{(2)}[p_1,...,p_{\ell}]\subset H^*(BT)\cong \bZ_{(2)}[t_1,...,t_{\ell}],\ |t_i|=2 \]
where the Pontryagin class $p_i$ is defined by
$\Pi_i(1+ t_i^2)=\sum_ip_i$.    Consider the maps
\[ \eta : T \stackrel{\eta_1}{\subset} U(\ell){\to} 
    SO(2\ell+1)\stackrel{\eta_2}{\to} U(2\ell+1).\]
Then $c_{2i}(\eta)=p_i\in CH^*(BT)^W$ which is the image
of $c_{2i}(\eta_2)\in CH^*(BSO(2\ell+1))$.  

On the other hand, $p_i=c_i(\eta_1)^2\ mod (2)$,
where $c_i(\eta_1)=\sigma_i$ is the elementary symmetric function in $S(t)$. 
Now we consider a versal torsor $\bG_k$ and the versal
flag $X=\bG_k/B_k$.  
From Lemma 6.6, the composition of the following maps
\[ CH^*(BG_k)/2\to CH^*(BB_k)/2\to CH^*(X)/2\]
is zero for $*>0$,  we get 
$ c_i(\eta_1)^2=\sigma_i^2=0$ in $CH^*(X)/2$.

This fact is also seen directly from considering
the natural inclusion $SO(2\ell+1)\to Sp(2\ell+1)$
and Proposition 6.2.

\begin{lemma}  We have $c_i^2=0$ in $CH^*(X)/2$.
\end{lemma}
\begin{lemma} There is an additive injection
\[  \bZ/2[c_1,...,c_{\ell}]/
(c_1^2,...,c_{\ell}^2) =\Lambda(c_1,...,c_{\ell})
\subset  CH^*(R(\bG_k))/2 .\]
\end{lemma}
\begin{proof}
At first we prove 
 that $c_1...c_{\ell}$
is nonzero in $CH^*(X)/2$.  Otherwise, it is
represented
by $2S(t)$ since $CH^*(X)$ is generated by elements from
$S(t)$.  It means that 
$2^{\ell-1}y_{top}=1/2(c_1...c_{\ell})\in S(t)$.  Hence 
$t(G)<2^{\ell}$ and a contradiction.

  For $I\subset (1,...,\ell)$, 
let $c_I=c_{i_1}...c_{i_k}$ and $y_I=y_{2i_1}...y_{2i_k}$
and $|I|=k$.
Suppose $c_I\in Ker(pr)$ for $pr:CH^*(X)/p\to CH^*(R(\bG_k)/p$.  Then from Corollary 5.2, we can write
\[   c_I=\sum_Jc_Ju(J)\quad with\ \  u(J)\in S(t)\ and\ |u(J)|>0,\]
since $c_I$ is not zero in $CH^*(X)/2$.
Then we have $2^{|I|}y_I=\sum_J2^{|J|}y_ju(j)$.  
Since $H^*(G/T)$ has no $2$-torsion,
dividing by $min(2^{|I|},2^{|J|})$, we have a contradiction to
$H^*(G/T;\bZ/2)\cong P(y)\otimes S(t)/(b)$.
\end{proof}
\begin{thm} Let $(G,p)=(SO(2\ell+1),2)$ and $X=\bG_k/B_k$.  Then there are  isomorphisms 
\[ CH^*(X)/2\cong S(t)/(2,c_1^2,...,c_{\ell}^2),\quad
CH^*(R(\bG_k))/2\cong \Lambda(c_1,...,c_{\ell}).\]
\end{thm}
\begin{proof}
We have the additive surjective map
\[ gr(S(t)/(2,c_1^2,...,c_{\ell}^2))\cong \Lambda(c_1,...,c_{\ell})\otimes S(t)/(c_1,...,c_{\ell})\]
\[ \twoheadrightarrow
 CH^*(X)/2\cong CH^*(R(\bG_k))\otimes 
S(t)/(2,c_1,...,c_{\ell}).\]
Therefore we see $CH^*(R(\bG_k))/2\cong \Lambda(c_1,,...,c_{\ell})$ from the preceding lemma.  From Lemma 7.4, we have the ring homomorphism
\[ S(t)/(2,c_1^2,...,c_{\ell}^2)\to CH^*(X)/2,
\]
which induces the ring isomorphism from the additive isomorphism.
\end{proof}
\begin{cor}
In the above theorem, $CH^*(X)$ is torsion free.
\end{cor}
\begin{proof}
Let us write 
$\Lambda_{\bZ}(a_1,...,a_m)=\bZ\{a_{i_1}...a_{i_s}|1\le i_1<...<i_s\le m\}.$
We consider the  
 restriction maps
\[ \begin{CD}
CH^*(R(\bG_k)\cong \Lambda_{\bZ}(c_1,...,c_{\ell})/J
 @>>> CH^*(\bar R(\bG_k))\cong
\Lambda_{\bZ}(y_2,...,y_{2\ell})\\
@VVV   @VV{inj.}V\\
CH^*(X)@>>> CH^*(\bar X).
\end{CD}\]
for some ideal $J$.
The first map is given by $c_i\mapsto 2y_{2i}$, 
and since the last map is a ring map, we see
$c_{i_1}...c_{i_s}\mapsto 2^sy_{i_1}...y_{i_s}$.  
Hence the first map 
is (additively) injective and $J=0$.
Hence $CH^*(R(\bG_k))$ is torsion free, and 
so is $CH^*(X)$ from the theorem by
Petrov-Semenov and Zainoulline
such that $M(X)_{(2)}\cong \oplus_i R(\bG_k)\otimes \bT^{i\otimes}$.
\end{proof}
{\bf Remark.}  The above lemmas, theorem and corollary are
also get from a result by Petrov (Theorem 1 in [Pe],
see also  Theorem 7.13 below).
\begin{cor}
Let $(G',p)=(SO(2\ell),2)$ and $X'=\bG_k'/B_k$
so that $G'\subset G=SO(2\ell+1)$.  Then
$t(G')=2^{\ell-1}$, and
\[CH^*(R(\bG_k'))/2\cong CH^*(R(\bG_k))/(2,c_{\ell})\cong \Lambda(c_1,...,c_{\ell-1}),\]
\[  CH^*(X')/2\cong CH^*(X)/(2,c_{\ell})
\cong S(t)/(2,c_1^2,...,c_{\ell-1}^2,c_{\ell}).\]
\end{cor}
\begin{proof}
This corollary is easily shown  from
$H^*(G';\bZ/2)\cong H^*(G;\bZ/2)/(y_{2\ell}).$
For example,  $grP(y)'\cong \Lambda(y_2,...,y_{2\ell-2})$
and $t(G')=2^{\ell-1}$.
\end{proof}
 \begin{cor} Let $G''=Sp(2\ell+1)$
and $X''=\bG_k''/B_k''$.  Then the natural maps
 $G\to G''\supset Sp(\ell)$ induce the isomorphisms
\[ CH^*(X)/2\cong CH^*(X'')/(2,t_i|i>\ell)
\cong H^*(Sp(\ell)/T;\bZ/2).\]
\end{cor}

We now study $CH^*(X|_K)/2$ for some interesting 
extension $K$ over $k$.
Let $K$ be an extension of $k$ such that
$X$ does not split over $K$ but splits over an extension
over $K$ of degree $2a$, $(a,2)=1$.  Suppose that
\[ (*)\quad y_{2i}\in Res_{K},\ for\ 1\le i\le \ell-1
\]
where $Res_{K}=Im(res: CH^*(X|_K)/2\to
CH^*(\bar X)/2)$.  We want to consider in the case $y_{2\ell}\not \in Res_K$.
\begin{lemma}  Suppose $(*)$ and $\ell\not =2^n-1$
for  $n>0$.
Then $y_{2\ell}\in Res_{K}$.
\end{lemma}
\begin{proof}
We see that if $\ell\not =2^n-1$, then each $y_{2\ell}$
is a target of the Steenrod operation.  Recall
   $ Sq^{2k}(y_{2i})={i\choose k}y_{2(i+k)}.$
It is well known that if $i=\sum i_s2^s$ and 
$k=\sum k_s2^s$ for $i_s,k_s=0$ or $1$, then  (in $mod(2)$)
\[ {i\choose k}={ i_m \choose k_m}... {i_s\choose k_s}...
              {i_0\choose k_0}.\]
Note that if $i=2^n-1$, then all $i_s=1$ (for $s<n$).
Otherwise there is $s$ such that $i_s=1$ but $i_{s-1}=0$.
Take $k=2^{s-1}$ and $i'=i-2^{s-1}$.  Then $i'+k=i$ and
\[ {i'\choose k}={ i_m=1 \choose 0}... {i'_s=0\choose k_s=0}{1\choose 1}...
              {i_0\choose 0}=1.\]
This means $Sq^{2k}(y_{2i'})=y_{2i}$ if $i\not= 2^n-1$.
\end{proof}
\begin{lemma}  Suppose $(*)$  and $\ell=2^n-1$.
Then elements $py_{2\ell},\ v_1y_{2\ell},\ ...,\ v_{n-1}y_{2\ell} $ are all in $ Im(res_{\Omega})$ 
where $res_{\Omega}:\Omega^*(X)/2\to \Omega^*(\bar X)/2.$
\end{lemma}
\begin{proof}
From Corollary 7.2,  we see
\[c_{\ell-2^j+1}=2y(2(\ell-2^j+2^0))+
v_1y(2(\ell-2^{j}+2^1))+...+v_j(y(2\ell))\]
\[ =v_j(y_{2\ell})\quad mod(y_{2},y_{4},...,y_{2\ell-2}).\]
Hence we have
$ res_{\Omega}(c_{\ell-(2^j-1)})=v_j(y_{2\ell})\ \ 
 mod(y_2,y_4,...,y_{2\ell-2}).$
\end{proof}
Thus we have
\begin{thm}  Suppose $(*)$ and $\ell=2^n-1$.
Then 
\[ CH^*(R(\bG_k)|_K)/2\cong 
\Lambda(y_2,...,y_{2\ell-2})\otimes CH^*(R_n)/2,\]
where $ CH^*(R_n)/2\cong \bZ/2\{1,py_{2\ell},v_1y_{2\ell},...,v_{n-1}y_{2\ell}\}.$
Moreover we have 
\[ Res_k^K(CH^*(R(\bG_k))/2)\cong 
\ CH^*(R_n)/2\subset CH^*(R(\bG_k)|_K)/2.\]
The restriction maps as 
$c_j\mapsto  v_sy_{2\ell}$ if $j=\ell-(p^s-1)$,
and $c_j\mapsto 0$ otherwise.
\end{thm}

At last of this section, we consider the case $X(\bC)=G/P$
with
\[ G=SO(2\ell+1)\quad and \quad P=U(\ell).\]
Let us write this $X$ by $Y$, i.e. $Y=\bG_k/P_k$.
From the cofibering $SO(2\ell
+1)\to Y(\bC)\to BU(\ell)$,  we have the spectral sequence
\[E_2^{*,*}\cong H^*(SO(2\ell+1);\bZ/2)\otimes H^*(BU(\ell))\]
\[\cong  P(y)\otimes \Lambda(x_1,...,x_{2\ell-1})\otimes \bZ/2[c_1,...,c_{\ell}]
\Longrightarrow H^*(Y(\bC);\bZ/2).\]
Here the differential
is given as $d_{2i}(x_{2i-1})=c_i$.  Hence
\[  CH^*(\bar Y;\bZ/2)\cong H^*(Y(\bC);\bZ/2)\cong   P(y).\]
This case is studied by Vishik [Vi] and Petrov [Pe]
as maximal orthogonal (or quadratic) grassmannian.
(see Theorem 5.1 in [Vi]).
From Theorem 7.6,  we have 
\begin{thm}  ([Vi],[Pe])  Let $G=SO(2\ell+1)$ and  $\bG_k$ be a versal $G_k$-torsor.  Let $Y=\bG_k/U(\ell)_k$.
Then
\[ CH^*(Y)/2\cong  CH^*(R(\bG_k))/2\cong \Lambda(c_1,...,c_{\ell}).\]
\end{thm}
{\bf Remark.} Petrov  computes the integral Chow ring for  
more general situations [Pe].
 From above theorem,  we note that $CH^*(R(\bG_k))/2$ has the ring structure in this case. 

In [Vi], 
Vishik originally defined the $J$-invariant $J(q)$
 of a quadratic
form $q$ which corresponds the quadratic grassmannian
(see Definition 5.11, Corollary 5.10 in [Vi])
by 
\[J(q)=\{i_k|y_{2i_k}\in Rec_{CH}\}\subset \{0,...,\ell\}.\]
Let $I$ be the fundamental ideal of the Witt ring $W(k)$ 
so that $grW(k)=\oplus_n I^n/I^{n+1}\cong K_*^M(k)/2$
where $K_*^M(k)$ is the Milnor $K$-theory of $k$.
Smirnov and Vishik (Proposition 3.2.31 in [Sm-Vi]) prove that
\[ q\in I^n \quad{ if\ and\ only\ if}\quad \{0,...,2^{n-1}-2\}\subset J(q). \]
Hence the condition $(*)$ in Theorem 7.11 is equivalent to $q\in I^n$ for the quadratic form $q$
corresponding to $Y|_K$. 
We also note that $G=Spin(m)$ cases correspond to
$q\in I^3$ from $1, y_2,y_4\in Res_{CH}$ (see
  (8.1) below). This fact is of course, well known.

\section{The spin group $Spin(2\ell+1)$ and $p=2$}

Throughout this section, let $p=2$,  $G=SO(2\ell+1)$
and $G'=Spin(2\ell+1)$.  By definition, we have the 
$2$ covering $\pi:G'\to G$.
It is well known that 
$\pi^*:   H^*(G/T)\cong H^*(G'/T')$ where $T'$ is a maximal torus of $G'$.  However the twisted flag varieties are
not isomorphic.  

Let $2^t\le \ell < 2^{t+1}$, i.e. $t=[log_2\ell]$.
The mod $2$ cohomology is
\[ H^*(G';\bZ/2)\cong H^*(G;\bZ/2)/(x_1,y_1)\otimes
\Lambda(z) \]
\[ \cong P(y)'\otimes \Lambda(x_3,x_5,...,x_{2\ell-1})\otimes \Lambda(z),\quad |z|=2^{t+2}-1\]
where
$P(y)\cong \bZ/2[y_2]/(y_2^{2^{t+1}})\otimes P(y)'$.  
(Here $d_{2^{t+2}}(z)=y_2^{2^{t+1}}$ for $0\not =y_2\in H^2(B\bZ/2;\bZ/2)$ in the spectral sequence induced from
the fibering $G'\to G\to B\bZ/2$.)
Hence
\[ (8.1)\quad grP(y)'\cong \otimes _{2i\not =2^j}\Lambda(y_{2i})\cong
\Lambda(y_6,y_{10},y_{12},...,y_{2\ell}).\]
The $Q_i$ operation for $z$ is given by Nishimoto [Ni]
\[ Q_0(z)=\sum _{i+j=2^{t+1},i<j}y_{2i}y_{2j}, \quad 
 Q_n(z)=\sum _{i+j=2^{t+1}+2^{n+1}-2,i<j}y_{2i}y_{2j}\ \ for\ n\ge 1.\]

We know that 
\[ grH^*(G/T)/2\cong P(y)'\otimes \bZ[y_2]/(y_2^{2^{t+1}})\otimes S(t)/(2,c_1,c_2,...,c_{\ell}) \]
\[ grH^*(G'/T')/2\cong P(y)'\otimes S(t')/(2,c_2',.....,c_{\ell}',c_1^{2^{t+1}}).\]
Here $c_i'=\pi^*(c_i)$ and $d_{2^{t+2}}(z)=c_1^{2^{t+1}}$ in the spectral sequence
converging $H^*(G'/T')$.)
These are additively isomorphic.  In particular, we have
\begin{lemma}  The element $\pi^*(y_2)=c_1\in S(t')$
and $\pi^*(t_j)=c_1+t_j$ for $1\le j\le \ell$.
\end{lemma}

Take $k$ such that $\bG_k$ is a versal $G_k$-torsor so that
$\bG'_k$ is also a versal $G_k'$-torsor.  Let us write
$X=\bG_k/B_k$ and $X'=\bG_k'/B_k'$.  Then
\[ CH^*(\bar R(\bG_k'))/2\cong P(y)'/2,\quad and \quad
    CH^*(\bar R(\bG_k))/2\cong P(y)/2.\]
\begin{thm} Let $(G,p)=(SO(2\ell+1),2),\
(G',p)=(Spin(2\ell+1),2)$, and $\pi:G'\to G$ be the natural projection.  Let $c_i'=\pi^*(c_i)$.  Then 
$\pi^*$ induces maps such that their composition map
is surjective
\[ CH^*(R(\bG_k)/(2,c_1)\cong 
 \Lambda(c_2,...,c_{\ell})\stackrel{\pi^*}{\to}
CH^*(R(\bG_k'))/2 \twoheadrightarrow
\bZ/2\{1,c_2',...,c_{\bar \ell}'\}\]
where $\bar \ell=\ell-1$ if $\ell=2^j$ for some $j>0$,
otherwise $\bar \ell=\ell$.
\end{thm}
\begin{proof}
From Corollary 5.3, we only need to show
$c_i'\not =0$ in
$\Omega^*(G_k'/T_k')/(I_{\infty}\cdot Im(res_{\Omega})).$
In fact, when $i\not =2^j$, in $H^*(G'/T')/4$, we have
\[ 2y_{2i}=c_j'\in S(t)\]
which is nonzero in $BP^*(G/T)/I_{\infty}\cdot Im(res_{\Omega}))$.  Because $y_{2i}\in P(y)'$ and $y_{2i}\not \in Im(res_{CH})$ from Lemma 4.5 since $X$ is a versal
flag variety.

When $i=2^j$,  we see $y_{2i}=y_{2^j}\in S(t')$, in fact
$y_{2^j}\not \in P(y)'$.   But in $BP^*(G'/T')/\II$,
we have
\[ 2y_{2i}+v_1(y(2i+2))+...+v_n(y(2i+2^{n+1}-2)+...= c_i'\in BP^*(BT').\]
When $i+1\le \ell$, 
this element is nonzero
in $BP^*(G/T)/I_{\infty}\cdot Im(res_{\Omega}))$ because  
\[ c_i'=v_1(y(2i+2))\not =0\in k(1)^*(G/T)/(v_1\cdot Im(res_{k(1)})\]
where $res_{k(1)}: k(1)^*(X')\to k(1)^*(\bar X))$.
Otherwise $y(2i+2)\in Im(res_{CH}))$, and this is a contradiction to $y_{2^j+2}\not \in Im(res_{CH})$, which follows from $y_{2j+2}\in P(y)'$ and 
Corollary 4.5.

When $2^j=\ell$, we note 
\[CH^*(\bar R( \bG_k'))/2\cong
CH^*(\bar R( \bG_k''))/2\quad  for\ G''=Spin(2\ell-1),\]
in fact $y_{2\ell}=y_{2^j}\not \in CH^*(\bar R(\bG_k'))$.
From a theorem by Vishik-Zainoulline (Corollary 6 in [Vi-Za]),
we get $CH^*(R(\bG_k'))/2\cong CH^*(R(\bG_k''))/2$.
Hence we can take $c_{\ell}'=0$.  
\end{proof}
%\begin{cor}
%The $2$-torsion ideal $Tor_2$ in $CH^*(X)_{(2)}$
%is generated by $c_{2^i}$ with $2\le 2^i\le \ell$.
%\end{cor}
%\begin{proof}
%Consider the restriction map
%\[ grCH^*(R(\bG_k))\supset \Lambda_{\bZ}(c_i|i\not= 2^j)\to 
%         grCH^*(\bar R(\bG_k))\cong \Lambda_{\bZ}(y_{2i}|i\not=
%2^j).\]
%This map (additively) injective.  Hence the submodule
%$\Lambda_{\bZ}(c_i|i\not =2^j)\subset CH^*(X)_{(2)}$ is
%$2$-torsion free.  
%

%On the other hand $c_{2^j}$ is $2$-torsion by the following reason.
%Note that $res_{\Omega}(c_{2^j})\in BP^{<0}\cdot\Omega^*(\bar X)$, and
%$res_{CH}(c_{2^j})=0 \in CH^*(\bar X)$.
%It is well known that  $res_{CH}\otimes \bQ$ is isomorphic.
%Hence $c_{2^j}$ must be torsion.
%\end{proof}

\begin{cor} The elements $c_{2^j}''=c_{2^j}'-c_1^{2^j}$, $j>0$ are torsion elements 
 in $CH^*(X)_{(2)}$.
\end{cor}
\begin{proof}
Note that $res_{\Omega}(c_{2^j}'')\in BP^{<0}\cdot\Omega^*(\bar X)$, and
$res_{CH}(c_{2^j}'')=0 \in CH^*(\bar X)$.
It is well known that  $res_{CH}\otimes \bQ$ is isomorphic.
Hence $c_{2^j}''$ must be torsion.
\end{proof}

{\bf Example.}
Let $G=SO(7)$ and $G'=Spin(7)$, i.e. $\ell=3$. Their cohomologies are
%\[ grH^*(G;\bZ/2)\cong \Lambda(x_1,x_2,...,x_6)\cong\]
\[H^*(G;\bZ/2)\cong  \bZ/2[y_2,y_6]/(y_2^4,y_6^2)\otimes \Lambda(x_1,x_3,x_5),\]
%\[ grH^*(G';\bZ/2)\cong \Lambda(x_3,x_5,x_6,z_7)\]
\[H^*(G';\bZ/2)
\cong \bZ/2[y_6]/(y_6^2)\otimes \Lambda(x_3,x_5,z_7).\]
The cohomologies of flag manifolds are
\[ H^*(G/T;\bZ/2)\cong \bZ/2[y_2,y_6]/(y_2^4,y_6^2)\otimes 
S(t)/(c_1,c_2,c_3),\]
\[ H^*(G'/T';\bZ/2)\cong \bZ/2[y_6]/(y_6^2)\otimes 
S(t)/(c_2',c_3',c_1^4).\]
These cohomologies are isomorphic by $\pi^*(y_2)=c_1$.
The torsion indexes are $t(G)=2^3$ and $t(G')=2$.
The Chow rings of versal flag varieties are
\[ CH^*(X)/2\cong S(t)/(2,c_1^2,c_2^2,c_3^2
), \quad CH^*(R(\bG_k))/2\cong \Lambda(c_1,c_2,c_3),\]
\[ CH^*(X')/2\cong S(t)/(2,(c_2')^2,c_2'c_3',(c_3')^2,c_1^4),\quad
 CH^*(R(\bG_k'))/2\cong \bZ/2\{1,c_2',c_3'\}.\]
Here $\pi^*(t_i)=c_1+t_i$ so that $\pi^*(c_1)=0\ mod(2)$.
For the third and the last isomorphisms, see Corollary 9.5
below.
In fact $G'$ is a group of type $(I)$.

\begin{lemma} (Marlin's bound) The torsion index $t(G')$ divides
$2^{\ell-[log_2\ell]-1}.$
\end{lemma}
\begin{proof} It follows from
\[\Pi_{i\not =2^j}c_i=\Pi_{i\not =2^j}(2y_{2i})=2^{\ell-t-1}y_{top}'\]
where $y_{top}'$ is the generator of top degree elements
in $P(y)'$.
\end{proof}
The exact value of $t(G')$ is determined Totaro, namely
$ t(G')=\ell-[log_2({ \ell+1
\choose 2}+1)] $
or that expression plus $1$. (It is known $t(Spin(2\ell+1)=t(Spin(2\ell+2))$.

Marlin's bound fails first for $Spin(11)$. This fact was first found by using a property of $12$-dimensional quadratic 
forms [To2].
However we show it using the
$Q_0$-operation. 
% It seems a new example of the Chow group
%of an irreducible motive which is not the original Rost %motives.
\begin{lemma}  For $(G',p)=(Spin(11),2)$, we have
$t(G')=2$ and the surjection
\[ CH^*(R(\bG_k'))/2\twoheadrightarrow  
\bZ/2\{1, c_2',c_3',c_4',c_5',c_2'c_4', c_1^8\}\]
\end{lemma}
\begin{proof}  Recall the cohomology
\[H^*(G';\bZ/2)\cong \bZ/2[y_6,y_{10}]/(y_6^2,y_{10}^2)\otimes \Lambda(x_3,x_5,x_7,x_9,z_{15}).\]
By Nishimoto, we know $Q_0(z_{15})=y_6y_{10}$. It implies
$2y_6y_{10}=d_{16}(z_{15})=c_1^8$.
Since $y_{top}'=y_6y_{10}$, we have $t(G')=2$.

%We can see $c_2c_3=0\in \Omega^*(\bar X)/(I_{\infty}\cdot %Res_{\Omega})$.  Because 
%$c_2c_3=2v_1y_6^2\in BP^*(G'/T')$ and
% $2a\in Res_{\Omega}=Im(\Omega^*(X)\to BP^*(G'/T')$ for %$a\in BP^*(G'/T')$, which implies $2y_6^2\in Res_{\Omega}$.  %The cases $c_2c_5=0, c_2c_1^8=0,...$ are  proved similarly.

We will show $c_2'c_4'\not =0\in CH^*(X)/2$. 
The elements $c_2',c_3',c_4'$ in
$CH^*(R(\bG_k')/2$ correspond to 
$v_1y_6,2y_6,v_1y_{10}$ in $\Omega^*(\bar R(\bG_k'))$ respectively.  In particular $c_2'c_4'$ corresponds
to $v_1^2y_6y_{10}$.
If $c_2'c_4'=0\in CH^*(R(\bG_k')))/2$, then 
$v_1y_6y_{10}$ must be in $Res_{\Omega}$.  
This means $v_1y_6y_{10}=b''$ for some $b''\in BP^*(BT')$.
However there is no $x\in H^{13}(G';\bZ/2)$ such that $Q_1(x)=y_6y_{10}$ with $d_{12}(x)=b''$.
\end{proof}

{\bf Remark.}  Quite recently, Karpenko showed that
the above surjection is an isomorphism.

In most cases, from the result Totaro, we see
$\Pi_{i\not =2^j}c_i'=0$.  However from [To2] when $\ell=8$, we know that
$2^{\ell-[log_2(\ell)]-1}=2^4=t(Spin(17))$.  
(Note   $y_{16}-2y_6y_{10}\in S(t)$ but $y_{16}\not \in S(t)$ when $\ell=8$.)  Hence we have
\begin{lemma}  Let $\ell\ge 8$ and $G'=Spin(2\ell+1).$
and $X'=\bG_k'/T'$.  Then  we have
\[ c_3'c_5'c_6'c_7',\ c_3'c_4'c_6'c_7'\not =0 \in CH^*(X')/2.\]
\end{lemma}
\begin{proof} It follows from that for $\ell=8$, elements
\[ c_3'c_5'c_6'c_7'=2^4y_6y_{10}y_{12}y_{14}=2^4y_{top}'
\quad and
\quad  c_3'c_4'c_6'c_7'=2^3v_1y_{top}'\]
are $BP^*$-module generators in $BP^*(G/T)/(I_{\infty}\cdot
Im(res_{\Omega}))$.
\end{proof}

\section{ The exceptional group $E_8$ and $p=5$}

In this section, we consider the case $(G,p)=(E_8,5)$.
The similar arguments also hold for $(G,p)=(G_2,2), (F_4,3)$.
The  $ mod(5)$ cohomology of $G=E_8$ ([Mi-Tod]) is given by
\begin{thm}
The $mod(5)$ cohomology
$H^*(E_8;\bZ/5)$ is isomorphic to 
\[ \bZ/5[y_{12}]/(y_{12}^5)\otimes 
\Lambda(z_{3},z_{11},z_{15},z_{23},z_{27},z_{35},z_{39},z_{47})\]
where suffix means its degree.  The cohomology  operations are given 
\[ \beta(z_{11})=y_{12},\ \ \beta(z_{23})=y_{12}^2, \ \ 
\beta(z_{35})=y_{12}^3,\ \ \beta(z_{47})=y_{12}^4, \ \ \]
\[P^1z_{3}=z_{11} ,\ \ P^1z_{15}=z_{23},\ \ P^1z_{27}=z_{35},\ \ 
 P^1z_{39}=z_{47}.\]
\end{thm}

We use the notation such that 
$y=y_{12}$ and $ x_1=z_3,...,x_8=z_{47}$
as used in $\S 2$.
Hence we can rewrite the cohomology as 
\[H^*(G;\bZ/p)\cong \bZ/p[y]/(y^p)\otimes \Lambda(x_1,...,x_{2p-2})\]
for $(G,p)=(E_8,5)$.  The above isomorphism also holds
for $(G,p)=(G_2,2),\ (F_4,3)$.  So hereafter in this 
section, we assume $(G,p)$ is one of $(G_2,2),\ (F_4,3)$
or $(E_8,5)$.
The cohomology operations are given as
\[ \beta: x_{2i}\mapsto y^i,\quad P^1:x_{2i-1}\mapsto x_{2i}
\qquad for\ 1\le i\le p-1.\]
Hence the $Q_i$ operations are given
\[ Q_1(x_{2i-1})=Q_0(x_{2i})=y^i\quad for\ 1\le i\le p-1.\]
Therefore we have the following lemma, by using Lemma 3.1
or Corollary 3.2. 
\begin{lemma}
In $BP^*(G/T)/\II$, we have 
\[ py^i=b_{2i}\quad mod(b_2,b_4,...,b_{2i-2})\cdot S(t),\]
\[ v_1y^i=b_{2i-1}\quad mod(b_1,b_2,...,b_{2i-3})\cdot S(t).\]
\end{lemma}
\begin{proof}
From Corollary 3.2, we see $py(2i)=b_{2i}\in BP^*(G/T)/\II$
and
\[ y(2i)=y^i+\sum_{j<i} y^jt(j)\quad where\ t(j)\in S(t)\ |t(j)|\ge 2.\]  By induction on $i$,  we get the first equation.
The second equations follows similarly, from
$v_1y(2i-1)=b_{2i-1}$ using Corollary 3.2.
\end{proof}

The fundamental class is written $y^{p-1}t_{top}\in H^{*}(G/T)$,
i.e., $y_{top}=y^{p-1}$.
Since $py^{p-1}=b_{2p-2}\in S(t)$, we see
$t(G)_{(p)}=p$.

 By Petrov-Semenov-Zainoulline,  it is
known when $G$ is one of $(G_2,2)$,$(F_4,3)$
or $(E_8,5)$,
the motive $R(\bG_k)$ in Theorem 4.2 is just the
original Rost motive $R_2$
defined by Rost and Voevodsky.
(Recall Theorem 4.6.)  The restriction  
$res_{\Omega|R}:\Omega^*(R(\bG_k))\to \Omega^*(\bar R(\bG_k))$
is injective.  Hence the following restriction is also injective
\[ res_{\Omega}: \Omega^*(X) \to \Omega^*(\bar X)\cong BP^*(G/T).\]
\begin{cor} 
We see
\[CH^*(R_2)/p\cong CH^*(R(\bG_k)/p\cong \bZ/p\{1,b_1,...,b_{2p-2}\}.\]
In particular, $b_s\not =0\in CH^*(X)/p).$
Moreover for  $1\le s,r\le 2p-2$, we  see
$b_sb_r=0$ in $CH^*(X)/p$.
\end{cor}
\begin{proof}
Recall Corollary 5.3.
We will prove $b_1\not =0\in CH^*(X)$, and the other cases are
proved similarly.
Note $b_1=v_1y\in \Omega^*(\bar X)$.  If $b_1\in BP^{<0}\cdot Im(res_{\Omega})$, then $y\in Im(res_{\Omega})$
and this is contradiction.  So $b_1\not =0$ in
\[CH^*(X)\cong \Omega^*(X)/(BP^{<0}\cdot \Omega^*(X))
\cong Im(res_{\Omega})/(BP^{<0}\cdot Im(res_{\Omega})).\]

For the last isomorphism,  we used the injectivity of $res_{\Omega}$.
We  prove $b_1^2=0\in CH^*(X)$. We see
\[ b_{1}^2=(v_1y)^2=v_1^2y^2=v_1b_3\in BP^{*}(G/T).\]
This element is contained in $BP^{<0}\cdot Im(res_{\Omega})$.  Hence it is zero in $CH^*(X)$ as above.
Other cases are proved similarly.
\end{proof}

\begin{thm}  Let $(G,p)=(G_2,2).\ (F_4,3)$ or $(E_8,5)$,  and
let $X=\bG_k/T_k$.
Then there is an isomorphism
\[CH^*(X)/p\cong S(t)/(p,b_ib_j|1\le i,j\le 2p-2).\]
\end{thm}
\begin{proof} 
From the preceding corollary we have the surjection 
\[ S(t)/(p,b_ib_j)\to  CH^*(X)/p.\]
On the other hand, it is immediate that there is an
 additive isomorphism
\[ S(t)/(p,b_ib_j)\cong \bZ/p\{1,b_1,...,b_{2p-2}\}\otimes
 S(t)/(p,b).\]
There is an injection  from the above right hand side module into $\Omega^*(\bar X)/(BP^{<0}\cdot 
Im(res_{\Omega}))$.  Hence  we have the theorem.
\end{proof}

{\bf Example.}
Let $G=F_4$ and $p=3$. We note $G''=Spin(9)\subset G$ and
\[ H^*(BG'')/3\cong H^*(BT'')^{W''}/3\cong \bZ/3[p_1,...,p_4]\]
for the Pontryagin classes  $p_i$ [Tod1].  So $H^*(G''/T'')/3\cong
S(t)/(3,p_1,...,p_4)$.  By using the induced map from $G''\subset G$,  we can see $b_i=p_i$ in $CH^*(X)/3$.  
Hence
\[ CH^*(X)/3\cong S(t)/(3,p_ip_j|0\le i,j\le 4).\]
%Moreover
%by Toda [Tod 1], it is known
%\[ Im(H^*(BG;\bZ/3)\to H^*(BT;\bZ/3))\cong H^*(BT;\bZ/3)^W\]
%\[ \cong \bZ/3[ p_1,p_2,\bar p_5,\bar p_9,\bar p_{12}]
%/(\bar r_{12}) \]
%for some $\bar p_i$ and $\bar r_{12}$, e.g.
%$\bar p_5=p_3(p_2-p_1^2)+p_4p_1$.  However it is known
%\[p_1,p_2,\bar p_5\not\in Im(i_{CH}^*:CH^*(BG)/3\to CH^*(BT)/3).\]
%So we see that the subalgebra $A$ of $S(t)/3$ generated by
%$p_ip_j$ for $0\le i,j\le 4$ contains $Im(i_{CH}^*)$.
%However note $\bar p_5\in A$ but $\bar p_5\not \in
%Im(i_{CH}^*)$.

Let $G'$ be of type $(I)$.  Then it is well known ([Mi-Tod])
that
there is a natural  embedding 
$i:G\subset G'$ where $(G,p)=(G_2,2), (F_4,3)$ or $(E_8,5)$
such that $i^*:H^*(G';\bZ/p)\to H^*(G;\bZ/p)$ is surjetive.
Moreover the polynomial rings  $P(y)$ and $P(y)'$
 are isomorphic by 
this map $i^*$.  This means $CH^*(\bar R(\bG_k)))\cong
CH^*(\bar R(\bG_k'))$.  This fact implies 
\[CH^*(R(\bG_k))\cong CH^*(R(\bG_k'))\]
 by  a theorem by Vishik and Zainoulline (Corollary 6 in [Vi-Za]).  Thus we have \begin{cor}  
Let $G'$ be of type $(I)$.  Then 
there are isomorphisms \[CH^*(R(\bG_k'))/p  \cong 
\bZ/p\{1,b_1,...,b_{2p-2}\},
\] \[CH^*(X')/p\cong S(t)/(p,b_ib_j,b_k|1\le i,j\le 2p-2,\ 
2p-1\le k\le \ell).\]
\end{cor}
\begin{proof}
We only need to show that we can take $b_k$ such that
$b_k=0\in CH^*(X')/p$.
Since $b_k=0$ in $BP^*(G/T)/I_{\infty}\cong H^*(G/T)/p$,
we can write
\[ b_k=\sum py^it(i)+\sum v_1y^it(i)'\quad in\ BP^*(G/T)/\II\]
where $t(i),t(i)'\in BP^*\otimes S(t)$.   Take
new $b_k$ by $b_k-\sum b_{2i}t(i)-\sum b_{2i-1}t(i)'$.
Then $b_k=0$ in $BP^*(G/T)/\II$.
\end{proof} 

{\bf Example.}
Recall the case $(G',p)=(Spin(7),2)$ and $(G,p)=(G_2,2)$.  Then we can take
$b_1=c_2'$, $b_2=c_3'$, and $b_3=c_1^4$, in fact
\[ CH^*(X')/2\cong S(t)/((c_2')^2,c_2'c_3',(c_3')^2,c_1^4),\quad
  CH^*(X)/2\cong  CH^*(X')/(c_1).\]

\section{The case $G=E_8$ and $p=3$}

In this section, we study in the case $(G,p)=(E_8,p=3)$.
  The cohomology $H^*(E_8;\bZ/3)$ is isomorphic to
([Mi-Tod]) 
\[ \bZ/3[y_{8},y_{20}]/(y_8^3,y_{20}^3)\otimes
\Lambda(z_3,z_7,z_{15},z_{19},z_{27},z_{35},z_{39},z_{47}).\]
Here the suffix means its degree, e.g., $|z_i|=i$.
By Kono-Mimura [Ko-Mi] the actions of cohomology 
operations
are also known
\begin{thm} ([Ko-Mi])
We have $P^3y_8=y_{20}$, and 
\[ \beta: z_7\mapsto y_8,\ \ z_{15}\mapsto y_8^2,\ \ 
z_{19}\mapsto y_{20},\ \
 z_{27}\mapsto y_{8}y_{20},\ \ z_{35}\mapsto y_{8}^2y_{20},\ \ 
z_{39}\mapsto y_{20}^2,\ \ z_{47}\mapsto y_8y_{20}^2\] 
\[ P^1:\ z_3\mapsto z_7,\ \  z_{15}\mapsto z_{19},\ \  z_{35}
\mapsto z_{39}\ \ P^3:\ z_7\mapsto z_{19},\ \  z_{15}\mapsto z_{27}
\mapsto -z_{39},\ \
 z_{35}\mapsto z_{47}.\]
\end{thm}
We use notations $y=y_8,y'=y_{20}$, and $x_1=z_3,...,x_8=z_{47}$.
Then we can rewrite the isomorphisms 
\[ H^*(G;\bZ/3)\cong \bZ/3[y,y']/(y^3,(y')^3)\otimes 
\Lambda(x_1,...,x_{8}).\]
\[ grH^*(G/T;\bZ/3)\cong \bZ/3[y,y']/(y^3,(y')^3)\otimes 
S(t)/( b_{1}, ,...,b_{8}).\]
From Lemma 3.4, we have 
\begin{cor} We can take $b_1\in BP^*(BT)$ such that 
 \[v_1y+v_2y'=b_1\quad in\  BP^*(G/T)/\II.\]
\end{cor} 

From the preceding theorem, we know that
all $y^i(y')^j$ except for $(i,j)=(2,2)$ are $\beta$-image.
Hence we have 
\begin{cor} 
For all nonzero monomials  $u\in P(y)/3$
except for $(yy')^2$,  it holds  $3u\in S(t)$.
That is, for $2\le k=i+3j+1\le 8$, $0\le i\le 2$, we can take  
\[ b_k=b_{i+3j+1}=3y^i(y')^j\quad in\ H^*(G/T)/(3^2).\]
   \end{cor}

\begin{lemma} Let $(G,p)=(E_8,3)$ and $X=\bG_k/T_k$.
In $BP^*(X)$, there are $b_i\in S(t)$ such that
 $b_i\not =0\in CH^*(X)/3$ and in $BP^*(G/T)/\II$
\[ b_k=b_{i+3j+1}=\begin{cases} v_1y+v_2y'\qquad if\   k=1\\
                             3y^i(y')^{j}\qquad if \ 0\le i\le 1,\ 2\le k\\
                     3y^2(y')^{j}+v_1(y')^{j+1} \qquad if\ i=2.
\end{cases}\]
\end{lemma}
\begin{proof}
Acting $r_{\Delta_1}$ on the equation $v_1y+v_2y'=b_1$
in $BP^*(X)/\II$, we have
\[ 3y+v_1r_{\Delta_1}(y)+v_2r_{\Delta_1}(y')=r_{\Delta_1}(b_1).\]
Note  $P^1(y),P^1(y')\in S(t)/3$ in $H^*(G/T;\bZ/3)$ since they are
primitive. 
Hence $v_1r_{\Delta_1}(y),v_2r_{\Delta_1}(y')\in BP^*\otimes S(t)$ $mod(\II)$.
So we have $3y=b_2$ in $BP^*(G/T)/\II$.
 Acting 
$r_{3\Delta_1}$ on the equation $3y=b_2 \in BP^*(X)/\II$,
we have $3y'=r_{3\Delta_1}(b_2)$, which is written by $b_3$.

Next we study the element $3y^2$ in $BP^*(X)/\II$.
Since $3y^2=b_3$ in $H^*(X)/(9)$, we have
\[  3y^2+v_1(a_1)
+v_2(a_2)=b_3\quad in\ BP^*(X)/\II.\]
Since $Q_1(x_3)=y'$, we can take  $a_1=y'$ by using the relation
$v_1y+v_2y'=b_1$. ( For example, when $a_1=y'+yb$, we use $v_1yb=-v_2y'b$.)  Since $v_2a_2$ is primitive in $k(2)^*(G/T)/(\II)$
(Recall the proof of Lemma 3.4), we can take $a_2=0$.
Otherwise if $a_2=\sum y^i(y')^jb$, for $i=1,2$, then 
\[v_2y^i\otimes (y')^jb\not =0\in k(2)^*(G)\otimes k(2)^*(G/T).\]
 Hence  we get
$ 3y^2+v_1y'=b_3\quad in\ BP^*(X)/\II.$

Acting $r_{3\Delta_1}$ and $r_{6\Delta_1}$ on the above
equation, we have the formulas for $yy'$ and $(y')^2$.
Here we used $r_{n\Delta_1}(y')\in BP^*(BT)/(\II)$.
since it is primitive.
Similar arguments work for the element $y^2y'$,
and we can see the formula for $y(y')^2$.
\end{proof}

\begin{cor}
The torsion index $t(E_8)_{(3)}=3^2$.
\end{cor}
\begin{proof}
The fundamental class (localized at $3$) is given as
$y_{top}t=y^2(y')^2t$ for some $t\in S(t)$.  Since
$b_2b_8=(3y)(3y(y')^2)=3^2y_{top} \in S(t)$, we see $t(E_8)_{(3)}=3$ or $3^2$.

Suppose $t(E_8)=3$, namely, $3y^2(y')^2=b'\in S(t)$.
From lemma 3.1, this implies that there is $x\in H^*(G;\bZ/3)$
such that $Q_0(x)=y^2(y')^2$ and $d_r(x)=b'$. But such $x$ does not exist from Theorem 11.2.
\end{proof}

Recall that
$ A_N=\bZ/3\{b_{i_1}\cdots b_{i_s}| |b_{i_1}|+...+|b_{i_s}|\le N\}.$
From Lemma 5.4, we have the surjection
$A_M\otimes S(t)/(b)\twoheadrightarrow CH^*(X)/3$
 for $M=|(yy')^2|=56.$  
\begin{thm}  Let $ (G,p)=(E_8,3))$ and $\bG_k$is a versal $G_k$-torsor. Then we have surjective maps  
\[ A_{56}\twoheadrightarrow 
CH^*(R(\bG_k))/3
\twoheadrightarrow  
\bZ/3\{1,b_1,...,b_8, b_1b_6,b_1b_8,b_2b_8\},
\]
\end{thm}
\begin{proof}
Since $t(E_8)_{(3)}=3^2$
and $X$ is a versal flag variety, we see
$3(yy')^2f\not \in res_{CH}$.  It follows $3(yy')^2\not \in
res_{CH}$.  Therefore $9(yy')^2,3v_1(yy')^2, 3v_2(yy')^2$ are $BP^*$-module generators in $Res_{\Omega}$, since 
\[res_{\Omega}\ :\ 
(b_2b_{8})\mapsto 9(yy')^2,\ \  
(b_1{b_8}) \mapsto 3v_1(yy')^2,\ \ 
(b_1b_{6})\mapsto 3v_2(yy')^2,\]
which shows they are nonzero.
\end{proof}
\begin{cor}
Let $Tor_3\subset CH^*(R(\bG_k))_{(3)}$ be the module
of $3$-torsion 
elements.  Then we have the  isomorphism
\[( CH^*(R(\bG_k))_{(3)}/Tor_3)\otimes \bZ/3\cong
\bZ/3\{1,b_2,...,b_8, b_2b_8\}. \]
\end{cor} 
\begin{proof}.  Let us write  by $b_i=py_{(i)}$ for $i\ge 2$.
Let $y_{(i)}y_{(j)}\not =y^2(y')^2$.  Then there is $k$ such that
$y_{(i)}y_{(j)}=y_{(k)}$.  Hence $b_ib_j=3b_k$ in $CH^*(\bar X)$.
So $b_ib_j-3b_k$ is a torsion element because
$res_{CH}\otimes \bQ$ is isomorphic.
\end{proof}

 We recall that there is an embedding $F_4\subset E_8$.
Let $K/k$ be a field extension
of degree $3a$ with $(3,a)=1$ such that
the flag variety  $X|_K=(\bG_k/T_k)|_K$ is still twisted but
$X|_{K'}$ is split for an extension $K'/K$ of degree
$3a'$ with $(3,a')=1$.
Note $P^3y=y'$ and if $y \in res_K^{\bar K}$, then so is $y'$.
Since $X|_K$ is twisted, we see $y'\in res_K^{\bar K}$ but
$y$ is not.  Hence the $J$-invariants are
\[ J(\bG_k|_K)=(1,0)\quad but \quad  J(\bG_k)=(1,1).\]
(See also 4.1.3 in [Pe-Se-Za], [Se] for $E_8$, $1\ge j_1\ge j_2$).

We know that the generalized Rost motive
for $F_4$ and $p=3$
 is just the original  Rost motive $R_2$.
Hence the natural map $i:F_4\to E_8$ induces the isomorphism
of motives over $\bar K$.
By Vishik-Zainouline ([Vi-Za]),
we have the isomorphism 
\[CH^*(R_2)/3\cong 
 CH^*(R(\bG_k|_K))/3.\]
\begin{prop}
Let us write the restriction map
\[ res_{k}^K: CH^*(R(\bG_k))/3\to CH^*(R(\bG_k)|_K)/3
\cong 
  CH^*(R_2)\otimes \bZ/3[y']/((y')^3).\]
Then we have
$ Im(res_{k}^{K})\cong \bZ/3\{1,b_1,b_2,b_3,b_5,b_6,b_8\}.$
\end{prop}
\begin{proof}
This proposition is proved by considering the restriction
on $\Omega^*(\bar X)$.  For example,
$b_8=3y(y')^2\not =0$ in $CH^*(X|_K)/3$, but
$b_2b_8=3\cdot(3y)(y')^2=0$.   In particular, we use 
the fact that 
 $b_4=3y',b_7=3(y')^2$ are in $Ker(res_k^K)$.
\end{proof}

\section{The case $G=E_8$ and $p=2$.}

In this section, we consider the case $(G,p)=(E_8,2)$.
The $mod(2)$ cohomology 
$H^*(E_8;\bZ/2)$ is given [Mi-Tod] as 
\[ \bZ/2[z_3,z_5,z_9,x_{15}]/(z_3^{16},z_5^{8},z_9^4,z_{15}^4)
\otimes \Lambda(z_{17},z_{23},z_{27},z_{29}).\]
Here we consider a graded algebra $grH^*(E_8;Z/2)$ identifying
$y_{2i}=z_i^2$ for $i=3,5,9,15$.

\begin{thm} 
The cohomology $grH^*(E_8;\bZ/2)$ is given 
\[ \bZ/2[y_6,y_{10},y_{18},y_{30}]
/(y_6^8,y_{10}^4,y_{18}^2,y_{30}^2),
\otimes \Lambda(z_3,z_5,z_9,z_{15},z_{17},z_{23},z_{27},z_{29}).\]
\end{thm}
Let us write $y_1=y_6,...,y_4=y_{30}$ and $x_1=z_3,x_2=z_5,...,x_8=z_{29}$. 
For ease of argument, let $x_4=z_{17}$ and $x_5=z_{15}$.
Hence we can write
\[ grH^*(E_8;\bZ/2)\cong \bZ/2[y_1,y_2,y_3,y_4]/(y_1^8,y_2^4,y_3^2,y_4^2)\otimes \Lambda(x_1,...,x_8).\]
\begin{lemma}  The cohomology operations acts as
\[\begin{CD}
  x_1=z_3 @>{Sq^2}>> x_2=z_5@>{Sq^4}>>
 x_3=z_9 @>{Sq^8}>> x_4=z_{17}\\
 x_5=z_{15} @>{Sq^8}>> x_6=z_{23}@>{Sq^4}>>
 x_7=z_{27} @>{Sq^2}>> x_8=z_{29}\\
 x_5=z_{15}  @>{Sq^2}>> x_4=z_{17} @. @. @.
\end{CD} \]
The Bockstein acts $Sq^1(x_{i+1})=y_i$ for $1\le i\le 3$,
$Sq^1(x_8)=y_4$  and
  \[Sq^1\ :\ x_5=z_{15}\mapsto y_1y_2,\ \  x_6=z_{23}\mapsto y_1y_3+y_1^4,\ \  x_{7}=z_{27}\mapsto y_2y_3.\]
\end{lemma}

Then we see from Lemma 3.4
\begin{cor}  In $BP^*(X)/\II$, we can take $y_1$ such that
for $r_{2\Delta_1}(y_1)=y_2$ and $r_{4\Delta_1}(y_2)=y_3$, 
we have  
\[ v_1y_1+v_2y_2+v_3y_3=b_1\quad
for \ b_1\in BP^*(BT).\]
\end{cor}

From Lemma 3.1 and the $Sq^1$ action in Lemma 11.2,
it is immediate that
\begin{lemma} Let $(G,p)=(E_8,2)$ and $X=\bG_k/T_k$.
In $H^*(X)/(4)$, there are $b_i\in S(t)$ such that
\[ b_k=\begin{cases}                              2y(1)\ 
(resp.\  2y(2),\ 2y(3))
\quad if \ k=2\ (resp. \ k=3,4)\\
                       2y(1,2)\ (resp.\ 2y(1,3),\
                        2y(2,3))  \quad  if\ k=5\ (resp.\ k=6,7)\\
                        2y(4)\quad if\ k=8.
\end{cases}\]
where 
$\pi^*y(i)=y_i$, $\pi^*y(i,j)=y_iy_j$ for the map $\pi:G\to G/T$.
\end{lemma}

We will study $b_i$ in more details by using 
the Quillen operation $r_{\alpha}$.  In particular recall $\rho r_{\alpha}(x)=\chi P^{\alpha}(\rho(x))$ for $\rho:BP^*(X)\to H^*(X;\bZ/2)$.
The anti-automorphism $\chi$ is defined by
\[ \chi(Sq^0)=Sq^0,\quad \sum_iSq^i\chi( Sq^{n-i})=0\ \ 
for \ n>0.\]
For example, (when $Sq^1=0$)  $\chi(P^i)=P^i$ for $i=1,2$ and 
$\chi(P^3)=P^2P^1$, ($P^3=P^1P^2$), and $\chi(P^4)=P^4+P^2P^2$.
\begin{lemma}
In $BP^*(X)/\II$, we have
\[ b_i=\begin{cases}
             2y_1+v_2(y_1^2)+v_3(y_2^2)\quad if\ i=2\\
            2y_2+v_1(y_1^2)+v_3(y_1^4)\quad if \ i=3\\
            2y_3+ v_1(y_2^2)             \quad if \ i=4
\end{cases} \]
\end{lemma}
\begin{proof}
On the equation $v_1y_1+v_2y_2+v_3y_3=b_1$,  act the
operation $r_{\Delta_1}$, and we get
\[ 2y_1+v_1r_{\Delta_1}(y_1)+v_2r_{\Delta_1}(y_2)+v_3r_{\Delta_1}(y_3)=r_{\Delta_1}(b_1).\]
Recall that $P^1(y_i)$ are primitive in $H^*(G/T;\bZ/2)$.  In fact, by
Kono-Ishitoya, we know  (Theorem 5.9 in [Ko-Is])
\[ P^1(y_1)\in S(t),\quad P^1(y_2)=y_1^2,\quad P^1(y_3)=y_2^2.\]
Thus we have
$2y_1+v_2(y_1^2)+v_3(y_2^2)=b_2.$
(Note also $Q_2(x_2)=y_1^2$, $Q_3(x_2)=y_2^2$.)

Acting $r_{2\Delta_1}$ on this formula,  we have 
\[ b_3=2y_2+v_1(y_1^2)+v_2(r_{\Delta_1}(y_1)^2)+v_3(r_{\Delta_1}(y_2)^2)\]
\[  =2y_2+v_1(y_1^2)+v_3(y_1^4).\]
Acting $r_{4\Delta_1}$, we have
$ b_4=2y_3+v_1y_2^2 $
where we used $P^2(y_1)=y_2$.
\end{proof}

From Lemma 3.1, we see that
\[ b_5=2y(1,2)=2(y_1y_2+\lambda y_1^2b')\quad where\
\lambda\in \bZ/2,\ b,b'\in S(t).\]
However, it is known the  more strong result
from Nakagawa [Na] and Totaro [To1].
\begin{lemma} ([Na], [To1])
In $H^*(G/T)/4$, we see $2y_1y_2\in S(t)$.
Indeed,  in the notation in  [To1]
\ \ $ d_8=1/9d_4^2-2/3g_3g_5$\ \ 
where $g_3=y_1$,  $g_5=y_2$
and $d_i\in S(t)$.
\end{lemma}

\begin{lemma}
In $BP^*(G/T)/(\II)$, we have, for some $b',b''\in S(t)$.
\[ b_5= 2y_1y_2+v_1(y_3)+v_2(y_2b'+y_2^2b'')+v_3(y_4).\]
\end{lemma}
\begin{proof}
From the preceding lemma,  we can write 
\[ b_5=2y_1y_2+v_1(a_1)+v_2(a_2)+v_3(a_3) \quad in\ BP^*(G/T)/(v_3,\II).\]
We may assume that $a_1$ does not contains $y_1$
by using the relation $b_1=v_1y_1+...$.
Note that in $k(i)^*(G/T)/\II$, each $v_ia_i$ are primitive.
Since $y_2$ is not in $Q_1$-image in $ H^*(G;\bZ/2)$, we see $y_2$ is $v_1$-torsion free in $k(1)^*(G)$.  So if $a_1$ contains $y_2$,
then $v_1a_1$ is not primitive in $k(1)^*(G)$, which is a contradiction. (E.g.,  if  $a_1=y_2y$, then $\mu^*(v_1a_1)=
v_1y_2\otimes y+...$)
 So $a_1$ contain only $y_3$, indeed $Q_1x_5=y_3$
implies $a_1=y_3$.  

For $a_2$, we know that
$y_1,y_3$ are not $v_2$-torsion.
Therefore $a_2$ only contains $y_2$, that is,
\[a_2=y_2b'+y_2^2b''\ mod(y_2^2)\quad for \ b',b''\in S(t).\]
By the primitivity in $k(3)^*(G/T)$, the element $a_3$
only
contains $y_3,y_4$.    We know $Q_3(x_5)=y_4$.
If $a_3=v_1(y_4+y_3b'')$, then let new $y_4$ be the element
$y_4+y_3b''$.
Thus we have the result.
\end{proof}
\begin{lemma} In $BP^*(G/T)/(\II,v_2,v_3)$, we have
\[b_6=2(y_1y_3+y_1^4+y_1^2b''')\quad for\ b'''\in S(t)\]
\end{lemma}
\begin{proof}
We consider the action $r_{4\Delta_1}$ 
on $b_5$.
By Cartan formula, we know
$r_{4\Delta_1}(y_1y_2)=\sum_i r_{i\Delta_1}(y_1)r_{(4-i)\Delta_1}(y_2)$. Here $r_{3\Delta_1}=\chi(P^3)=P^2P^1\ mod(2)$.  Hence
we have  with $mod(2)$
\[r_{3\Delta_1}(y_1)r_{\Delta_1}(y_2)=P^2P^1(y_1)P^1(y_2)
=by_1^2,\]
\[r_{2\Delta_1}(y_1)r_{2\Delta_1}(y_2)=
y_2b'', and \  \ 2y_2b''\in S(t),\]
\[r_{\Delta_1}(y_1)r_{3\Delta}(y_2)=
    b''b'''\in S(t),\quad and \quad
   y_1r_{4\Delta_1}(y_2)=y_1y_3.\]
Hence $r_{4\Delta_1}(y_1y_2)=y_1y_3+by_1^2\ mod (BP^*\otimes S(t))$.

Next consider
\[ r_{4\Delta_1}(v_1y_3)=2r_{3\Delta}(y_3)+v_1(r_{4\Delta_1}(y_3)).\]
Here with $mod(2)$ we see 
\[ r_{3\Delta}(y_3)=P^2P^1(y_3)=P^2(y_2^2)=y_1^4.\]
We also see $r_{4\Delta_1}(y_3)=P^4y_3\in S(t)$ from the
primitivity in $H^*(G/T;\bZ/2)$.

At last we can see  
\[ r_{4\Delta_1}v_2(b'y_2+b''y_2^2)=
    v_1r_{2\Delta_1}(b'y_2+b''y_2^2) =0\quad mod(v_2).\]
Because if its contains $v_1y_2$ or $v_1y_2^2$,
then it contradicts to the primitivity of $k(1)^*(G/T)$.
If it contains $v_1y_1$, then it is in $Ideal(v_2)$ by the relation
$b_1$.
Thus we have the result ($with\ mod(v_2,v_3)$) of this lemma.
\end{proof}
Similarly considering $r_{2\Delta}(b_6)$ and $Q_1x_7=y_4$, we have 
\begin{lemma}
In $H^*(G/T)/(I_{\infty}, v_2,v_3)$, we have
$ b_7=2(y_2y_3+by_1^2+b'y_2^2)+v_1y_4$.
\end{lemma}

{\bf Remark.}  For the preceding two lemmas,
Totaro gets more strong and explicit  results.
Totaro (Lemma 4.4 in [To1])
shows in $H^*(G/T)_{(2)}$
\[ d_6^2-25/81d_4^3+2(15g_9g_3+1/3g_3^4-5/3g_5d_7-125/9g_5
g_3d_4)\]
\[\qquad +2^2(-23/3g_3^2d_6)=0\]
where $g_3=y_1,g_5=y_2,g_9=y_3$ and $d_i\in S(t)$.
This implies $2(y_1y_2+y_1^4)\in S(t)$.
Therefore we can take $b'''=0$ in Lemma 11.8.
Totaro also gives explicit formula $d_7,d_8$ in $H^*(G/T)$.
In particular, in Lemma 4.4 in [To1]), he shows
$b=b'=0$ in the above lemma.  

At last, from $\beta(x_8)=y_4$, we note
\begin{lemma}
In $H^*(G/T)/4$, we see $2y_4=b_8$.
\end{lemma}

Now we study the torsion index.  Recall
\[ y_{top}=\Pi _{i=1}^{4}y_i^{2^{r_i}-1}=y_1^7y_2^3y_3y_4\in P(y)\]
and $t_{top}$ are top degree elements in $P(y)$ and
$S(t)/(b)$ so that $f=y_{top}t_{top}$ for the fundamental 
class $f$ of $H^*(G/T)_{(2)}$.
\begin{lemma} (Totaro [To1]) We have $t(E_8)_{(2)}=2^6$.
\end{lemma}
\begin{proof}
We consider the element
\[\tilde b=b_5^3b_6b_4b_8 =2^6(y_1y_2)^3(y_1y_3+y_1^4+y_1^2b'')(y_3)(y_4).\]
Here using $y_3^2=b'\in S(t)\ mod(2)$, we have
\[ (y_1y_3+y_1^4+y_1^2b'')y_3=y_1^4y_3+(y_1b'+y_1^2b'')y_3.\]
Hence we can write
\[ \tilde b=2^6(y_{top}+\sum yt)\quad for \  |t|>0.\]
From  Lemma 5.5, we see $t(E_8)_{(2)}\le 2^6$.

Suppose $t(E_8)_{(2)}\le 2^5$, that is, $2^5f=2^5y_{top}t_{top}\in S(t)$.  Then $2^5f$ must be in ideal $I=(b_1,...,b_8)$, and we can write for $b_i=2y_{(i)}$ (note  $y_{(1)}=0$, and $y_{(i)}$ is not a monomial, in general)
\[(*)\quad  2^5f=\sum b_it(i)=2\sum y_{(i)}t(i)\quad for\  t(i)\in S(t).\]
Since $H^*(G/T)$ has no torsion, we have
$ 2^4f=\sum y_{(i)}t(i).$

Let us rewrite $s=\sum y_{(i)}t(i)=\sum_I y^It(I)$ for a
monomial $y^I=y_1^{i_1}...y_4^{i_4}\in P(y)$ for $I=(i_1,...,i_4)$,
and $t(I)\in S(t)$.
Then $s\in Ideal(2)$ implies each $t(I)\in Ideal(b_1,...,b_8)
\subset S(t)$,  since $H^*(G/T)/2\cong P(y)\otimes S(t)/(b)$.
Continue this arguments, we have
\[ (**)\quad f=\sum y^It(I)\quad in\ H^*(G/T).\]
Consider this equation in $H^*(G/T)/2$, and we
see $f=\sum y^It(I)$, that is $y^I=y_{top}$ and $t(I)=t_{top}$.

To get $(**)$ from $(*)$, we exchanges $b_i$ to $2y_{(i)}$
at most five times.
 
We note that  the largest number of $y_i$'s of monomials 
in $y_{(j)}$ is $1$ or $2$ except for $y_{(6)}=(y_1y_3+y_1^4+y_1^2b)$ where the number is $4$. We easily see that $y_{(6)}$ appears as $y_{(i)}$ just one time in the process $(*)$ to $(**)$. 
 We also  see that $y_{(i)}=y_{(8)}$ just one time for the existence of $y_4$.  Let us write by $\sharp_y(a)$ the number of $y_i$'s in $a$.  Then
\[ \sharp_y(y_{(i_1)}...y_{(i_5)})\le 
2\times 3+4+1=11.\]
On the other hand $\sharp_y(y_{top})=7+3+1+1=12$.
Thus $t(E_8)_{2}\ge 2^6$.
\end{proof} 
\begin{lemma}
Let $(i_1,...,i_k)\subset (4,5,5,5,6,8)$.  Then
$\tilde b=b_{i_1}...b_{i_k}\not =0$ in
$CH^*(X)/2$ since $b_5^3b_4b_6b_8\not =0$.
\end{lemma}

Let $K$ be an extension of $k$ such that
$X$ does not split over $K$ but splits over an extension
over $K$ of degree $2a$, $(2,a)=1$.  Suppose that
\[ (*)\quad y_1,y_2,y_3 \in Res_{K},\quad but \quad
y_4\not \in Res_{K}\]
where $Res_{K}=Im(res: CH^*(X|_K)/2\to
CH(\bar X)/2)$.  
(Compare the condition $(*)$ in $\S 7$.)
That is, $J(\bG_k|_K)=(0,0,0,1)$
and such $K$ exists (see [Pe-Se-Za], [Se]).
Then we have the following theorem by arguments similar to those to get Theorem 7.12.
(The motive $R(\bG_k)|_K$ in the theorem 
is an example of motives 
given in Lemma 8.4 in [Se].)
\begin{thm}  There is an isomorphism   
\[ CH^*(R(\bG_k)|_K)/2\cong 
\bZ/2[y_1,y_2,y_3]/(y_1^8,y_2^4,y_3^2)\otimes CH^*(R_4)/2,\]
where $ CH^*(R_4)/2\cong \bZ/2\{1,2y_4,v_1y_4,v_2y_4,v_{3}y_4\}.$
We have 
\[ Res_k^K(CH^*(R(\bG_k))/2)\cong CH^*(R_4)/2\subset
  CH^*(R(\bG_k)|_K)/2.\]
The restriction map is given as
$ b_j\mapsto v_{8-j}y_{4}$ if $5\le j \le 8$, and 
$b_j\mapsto 0$ if  $1\le j\le 4$.
\end{thm}

\section{ The exceptional group  $E_7$ and $p=2$}

The $mod(2)$ cohomology of $E_7$ is given
\[H^*(E_7;\bZ/2)\cong H^*(E_8;\bZ/2)/(z_{3}^4,z_{5}^4,z_{15}^2,z_{29}).\]
We use the notations in the preceding sections.
\begin{thm}  We have an isomorphism
\[ grH^*(E_7;\bZ/2)\cong 
 \bZ/2[y_1,y_2,y_3]/(y_1^2,y_2^2,y_3^2)\otimes\Lambda(x_1,...,x_7),\] 
where $i^*(y_j)=y_j$ for $1\le j\le 3$ and $i^*(x_i)=x_i$ for $1\le i\le 7$
and $i^*(y_4)=i^*(x_8)=0$ for the natural embedding 
$i: E_7\subset E_8$
\end{thm}
\begin{cor}
  In $BP^*(X)/\II$, we can take $y_1$ such that
for $r_{2\Delta_1}(y_1)=y_2$ and $r_{4\Delta_1}(y_2)=y_3$, 
it holds   
$ v_1y_1+v_2y_2+v_3y_3=b_1$ for
$\ b_1\in BP^*(BT).$
\end{cor}
From Lemma 3.1 and the $Sq^1$ action in Lemma 11.2,
it is immediate
\begin{lemma} Let $(G,p)=(E_7,2)$.
In $H^*(G/T)/(4)$, for all monomials $u\in P(y)/2$,
except for $y_{top}=y_1y_2y_3$, the elements $2u$ 
are written as elements in
$H^*(BT)$.
Namely, in $H^*(G/T)/(4)$, there are $b_i\in S(t)$ such that
\[ b_k=\begin{cases}                              2y_{1}\ (resp.\  2y_2,\ 2y_3)
\quad if \ k=2\ (resp. \ k=3,4)\\
                       2y_1y_2\ (resp.\ 2y_1y_3,\
                        2y_2y_3)  \quad  if\ k=5\ (resp.\ k=6,7).
                       \end{cases}\]
\end{lemma} 

From lemma 11.5, it is immediate
\begin{lemma}
In $BP^*(X)/\II$, we have
$ 2y_1=b_2, 2y_2=b_3,  2y_3=b_4.$
\end{lemma}
\begin{lemma} We have $t(E_7)_{(2)}=2^2$.
\end{lemma}
\begin{proof}
We get the result from $b_2b_7=(2y_1)(2y_2y_3)=2^2y_{top}$.
\end{proof}

\begin{cor}  There are  surjective maps 
\[A_{34}\twoheadrightarrow CH^*(R(\bG_k))/2
\twoheadrightarrow  
\bZ/2\{1,b_1,...,b_7, b_1b_5,b_1b_6,b_1b_7, b_2b_7\}.
\]
\end{cor}
\begin{proof}  Note that 
 $|y_1y_2y_3|=34$.
In $\Omega^*(\bar  X)$, we see
\[b_1b_5=2v_3y_{top},\quad  b_1b_6=2v_2y_{top},
\quad b_1b_7=2v_1y_{top}.\]
These elements are $\Omega^*$-module generators
in $Im(res_{k}^{\bar k}(\Omega^*(X)\to \Omega(\bar X))$
because $2y_1y_2y_3\not \in Im(res_k^{\bar k})$
from the fact $t(\bG_k)=2^2$.
\end{proof}
By the arguments similar to Corollary 10.7,  we have
\begin{cor} Let $Tor_2\subset CH^*(R(\bG_k)_{(2)}$
be the modules of $2$-torsion elements.  Then we have
an isomorphism 
\[ CH^*(R(\bG_k))/(2,Tor_2)\cong \bZ/2\{1,b_2,...,b_7,b_2b_7\}.\]
%\[ Tor_2\twoheadrightarrow 
%\bZ/2\{b_1,b_1b_5,b_1b_6,b_1b_7\}.\]
\end{cor}

Let us write $G'=E_8$ and $G''=G_2$ so that
$G''\subset G=E_7\subset G'$. Take fields 
$k\subset K\subset K'$ such that
\[ (**)\quad y_1^2,y_2^2,y_4\in Res_K\quad but \quad 
       y_1,y_2,y_3\not \in Res_K,\]
\[
(***)\quad y_1^2,y_2,y_3,y_4\in Res_{K'}\quad
but \quad y_1\not \in Res_{K'}.\]
Then the following proposition is almost immediate
\begin{prop}  Suppose $(**)$ and $(***)$.
We have isomorphisms,
\[ CH^*(R(\bG_k')|_K)/2\cong \bZ/2[y_1^2,y_2^2,y_4]/(y_1^8,y_2^4,y_4^2)\otimes
CH^*(R(\bG_K))/2,\]
the restriction is given by $b_i\mapsto b_i$
for $1\le i\le 7$ and $b_8\mapsto 0$, and 
\[CH^*(R(\bG_K)|_{K'})/2\cong \bZ/2[y_2,y_3]/(y_2^2,y_3^2)\otimes CH^*(R_2)/2,\]
the restriction  is given by $b_i\mapsto b_i$ for $i=1,2$,
and $b_i\mapsto 0$ for  $3\le i\le 7$.
\end{prop}

\end{document}